\documentclass[a4paper,12pt]{amsart}

\usepackage{amsmath,amsfonts,amsthm,amsopn,cite,mathrsfs,amssymb}
\usepackage{mathtools}
\usepackage{epsfig,verbatim}
\usepackage{subfigure}
\usepackage{xcolor}

\usepackage[a4paper,left=3cm,right=3cm,top=3.5cm,bottom=3.5cm]{geometry}

\newcommand{\bE}{\mathbb{E}}\newcommand{\N}{\ensuremath{\mathbb{N}}}

\newcommand{\hs}{\mathbb{H}^{s}}
\newcommand{\mz}{Marcinkiewicz-Zygmund}

\newcommand{\field}[1]{\mathbb{#1}}
\newcommand{\bR}{\field{R}}        
\newcommand{\bN}{\field{N}}        
\newcommand{\bC}{\field{C}}        
\def\rd{\bR^d}
\newcommand{\modsp}{modulation space}
\newcommand{\tfa}{time-frequency analysis}

\newcommand{\stft}{short-time Fourier transform}
\def\cF{\mathcal{ F}}              

\def\cH{\mathcal{ H}}
\def\cO{\mathcal{ O}}
\def\inv{^{-1}}
\DeclareMathOperator*{\supp}{supp}
\def\lrd{L^2(\rd)}

\newtheorem{tm}{Theorem}[section]    
\newtheorem{lemma}[tm]{Lemma}
\newtheorem{prop}[tm]{Proposition}

\theoremstyle{definition}

\newtheorem{rem}{Remark}[section]

\newcommand{\fif}{if and only if}

\DeclareMathOperator*{\spann}{span}
\begin{document}
\begin{abstract}
We study Gauss quadrature for Freud weights and derive worst case
error estimates for functions in a family of associated Sobolev
spaces. For the Gaussian weight $e^{-\pi x^2}$ these spaces coincide
with a class of modulation spaces which are well-known in
(time-frequency) analysis and also appear under the name of  Hermite
spaces. Extensions are given to more general sets of nodes that are
derived from \mz\ inequalities. This generalization can be
interpreted as a stability result for Gauss quadrature. 
\end{abstract}

\title[Gauss Quadrature for Freud Weights and  Modulation Spaces]{Gauss Quadrature for Freud Weights, Modulation Spaces, and
  Marcinkiewicz-Zygmund Inequalities}
\author{Martin Ehler}
\author{Karlheinz Gr\"ochenig}
\address{Faculty of Mathematics \\
University of Vienna \\
Oskar-Morgenstern-Platz 1 \\
A-1090 Vienna, Austria}
\email{martin.ehler@univie.ac.at}
\email{karlheinz.groechenig@univie.ac.at}
\date{}
\thanks{K.\ G.\ was
  supported in part by the  project P31887-N32  of the
Austrian Science Fund (FWF)}
\maketitle

\section{Introduction}
The fame of Gauss quadrature partially comes from the connection to
the theory of orthogonal polynomials and the fact that a
quadrature rule based on $n$ points in $\bR $ is exact for
polynomials of degree $\leq 2n-1$. Although this is standard material in undergraduate text books on numerical analysis, 
Gauss quadrature, and quadrature rules in general,  remain an active
topic of research. In this paper we study the   quadrature
that is associated to  the
orthogonal system of Hermite polynomials. This so-called Gauss-Hermite
quadrature 
is a quadrature rule for integrals on the whole real line, whose
success and efficiency is not uncontested 
in the literature. Therefore Gauss-Hermite quadrature is the subject
of many recent investigations. 

The basic results for Gauss-Hermite quadrature are  explicit error bounds in classical
style involving a weighted norm of the derivative of the
function~\cite{Mastroianni:1994aa}. However,  Trefethen~\cite{Trefethen:aa} argues that a
significant portion of the nodes and weights is  below machine
precision and thus  numerically irrelevant,   and
he suggests to ignore large  nodes.  

In several recent papers, 
Dick, Irrgeher, Pillichshammer et al.\ in various
combinations~\cite{Irrgeher:2015ab,Irrgeher:2015aa,IKPW16,Dick:2018aa} have
introduced   function spaces on $\bR $ and $\rd $, which
they call  Hermite spaces, for which they  derive strong error
estimates. In addition, these authors show that in certain cases
digital nets outperform Gauss-Hermite nodes~\cite{Dick:2018aa}.   

Misspecified settings are considered in \cite{Kanagawa:2020aa}, and
adaptive Gauss-Hermite quadrature is discussed in
\cite{Jin:2020aa,Liu:1994aa}. Recently, the trapezoidal rule
has shown better performance than Gauss-Hermite nodes, see
\cite{Kazashi:2022mn}. Further integration problems with
Gauss-Hermite nodes are discussed in \cite{Smith:1983pt}, and, for
hyperinterpolation on general domains, we refer to
\cite{Sloan:1995bx}.

With this background we take another look at Gauss-Hermite quadrature
and its generalizations. Our specific contributions can be summarized
as follows: 

(i) We study Gauss quadrature and error estimates  for Gauss
quadrature with respect to  the Freud
polynomials which are orthogonal with respect to the weight function
$e^{-2\pi |x|^\alpha }$ for $\alpha >1$. This quadrature includes in particular
the Gauss-Hermite quadrature.

(ii) We introduce a  scale of function spaces that are adapted
to the weight and lead to natural error estimates for Gauss
quadrature. These function spaces can be viewed as the appropriate
versions of Sobolev spaces in the context of orthogonal
polynomials. The motivation comes from the study of \mz\ inequalities
in~\cite{Lubinsky:1996aa,Filbir:2011fk,Grochenig:2019mq,Mhaskar:2002ys} and the work of Irrgeher et al.
~\cite{Irrgeher:2015ab,Irrgeher:2015aa,IKPW16,Dick:2018aa}. For the Gaussian
weight $e^{-\pi |x|^2}$  the latter
authors have introduced these spaces (Hermite spaces in their
terminology) to study quadrature rules from the point of view of
complexity theory.

As a typical result we offer a (slightly vague) version of the main
theorem: \emph{ Let $W(x) = e^{-\pi |x|^\alpha }$ for $\alpha >1$ and let
$H_n$ be the orthogonal polynomials with respect to the weight
$W(x)^2$. Then the functions $h_k = H_k  W$ form an orthonormal basis for
$L^2(\bR )$. The associated Sobolev space  $\hs $ is defined by the
norm $\|f\|_{\hs } ^2 = \sum_{k=0}^\infty |\langle f, h_k \rangle |^2
\, (1+k)^s$. Next assume that the set of nodes  $X_n = \{ x_1, \dots ,
x_N\} \subseteq \bR  $ satisfies the \mz\ inequalities
(with weights $\tau (x_j)$) for the finite-dimensional
subspace $\Pi _n = \mathrm{span}\, \{ h_0, \dots , h_n\}$, i.e.,
$$
a \|f\|^2_{2}\leq \sum _{j=1}^N |f(x_j)|^2 \tau (x_j) \leq b\|f\|^2_{2} ,\qquad \text{ for all
} f\in \Pi_{n} \, .
$$
Under a small technical assumption on the spread of $X_n$, we prove
the existence of  weights $\omega _j\in \bR, j=1, \dots ,N$ such that  }
\begin{equation*}
\sup_{\substack{f\in \mathbb{H}^s\\ \|f\|_{\mathbb{H}^s}\leq
    1}}\left|\int_{-\infty}^\infty
  f(x)W(x)\mathrm{d}x-\sum_{j=1}^N\omega_j f(x_j) \right|^2   \lesssim
n^{-s+\frac{4}{3}}\, .
\end{equation*}
 This may also be read as a stability result of Gaussian
 quadrature. If $\{x_j\}_{j=1}^n$ are small distortions of the
 Gaussian quadrature nodes, then $\{\omega\}_{j=1}^n$ are only  small
 distortions (hence positive) of the associated Gaussian 
 quadrature weights. In general and in contrast to Gauss quadrature the number of nodes $N$ need not
 be identical to the dimension of $ \Pi _n$. Our proof technique is applicable in much more
 general settings and the term $\frac{4}{3}$ in the error estimate is
 probably  an artifact of this
 generality.  

(iii) For $\alpha = 2$, i.e., for the Gaussian weight $e^{-\pi x^2}$
we show that the  Hermite spaces of ~\cite{Irrgeher:2015ab,Irrgeher:2015aa,Dick:2018aa} coincide with a
class of well-known function spaces in analysis, namely the \modsp s
introduced by Feichtinger~\cite{fei83,fei06} in 1983. Thus the 
abstract definition of ``Hermite space'' or ``appropriate Sobolev
space''  turns out to be perfectly natural and leads to a class of
well-studied and important function spaces. In fact, the Hermite
spaces aka \modsp s are precisely the Sobolev spaces with respect to
the harmonic oscillator and known as Shubin classes in the PDE
literature~\cite{shubin}. 
We feel that these
identifications are  of independent interest, as they  add  new tools  to
the study of numerical quadrature related to
Hermite polynomials. We mention that \modsp s are the canonical
function spaces in \tfa\ ~\cite{book}. They  have many applications in the
analysis of pseudodifferential operators and even in wireless
communications~\cite{book,gro06,CR20} and  are used for the analysis of
nonlinear PDE~\cite{BO20}, the formulation of uncertainty principles, in the
theory of coherent states (Gabor analysis) and many more, see also \cite{gabbook1,gabbook2}. With some
satisfaction, we may now say that \modsp s are also useful  in
numerical analysis.

(iv) We extend the error analysis to general sets of nodes beyond
Gaussian nodes (which are the zeros of the $n+1$-st orthogonal
polynomial). The advantage is that one can build in redundancy so that
a missing node does not spoil the result of the numerical quadrature.

\vspace{3mm}

On the technical level, we highlight the role of \mz\ inequalities in
the derivation of quadrature nodes and weights. A fundamental tool are
precise and deep estimates from~\cite{Levin:1998wj,Levin:1992mb} for the size of orthogonal polynomials with
respect to Freud weights and for the associated Christoffel
functions.  Specifically, lower bounds for the Christoffel function
are used to derive tail estimates for the reproducing kernel of the
associated Sobolev spaces.

\bigskip
 The outline is as follows: In Section \ref{sec:function class}, we
 collect the basic definitions, introduce the class of orthogonal
 polynomials and  the associated Sobolev-type spaces.
  We derive some preparatory inequalities for the reproducing kernels
  in these spaces that are used in  Sections \ref{sec:GQ} and \ref{sec:MZ}. We establish error
 bounds for Gaussian quadrature nodes in Section \ref{sec:GQ}. In
 Section \ref{sec:modulation}, we prove  that the Hermite spaces of
 ~\cite{Irrgeher:2015ab,Irrgeher:2015aa,Dick:2018aa}  coincide with
 classical modulation spaces. Some numerical 
 experiments for modulation spaces with exponential weights are
 presented in Section \ref{sec:nums}. Sections \ref{sec:GHS} and
 \ref{sec:MZ} are dedicated to error bounds beyond Gaussian
 nodes. We recast the quadrature problem by means of general Hilbert
 spaces in Section \ref{sec:GHS}. The actual generalization beyond
 Gaussian nodes is derived in Section \ref{sec:MZ}.

\bigskip
\section{Sparsity classes for Hermite type functions}\label{sec:function class}
For $\alpha > 1$, we consider the Freud weight function
\begin{equation*}
W(x)=e^{-\pi |x|^\alpha},\qquad x\in\mathbb{R}.
\end{equation*}
The  orthonormal polynomials with respect to $W^2$ are called Freud
polynomials which we  denote by $\{H_{k}\}_{k\in\mathbb{N}}$,  so that
$\deg(H_k)=k$. Consequently the functions 
\begin{equation*}
h_{k}:=H_{k} W, \qquad k\in\mathbb{N},
\end{equation*}
form  an orthonormal basis for $L^2(\mathbb{R})$. For $f \in
L^2(\mathbb{R})$, we denote  the coefficient with respect to this basis by
\begin{equation*}
\hat{f}_{k} :=\langle f,h_{k}\rangle_{L^2(\mathbb{R})},\qquad k\in\N.
\end{equation*}
Fix $s\geq 0, p,q>0$ and consider the inner products 
\begin{equation*}
 \langle f,g\rangle_{\mathbb{H}^s}=\sum_{k\in\mathbb{N}} (1+k)^s
 \hat{f}_{k}\, \overline{\hat{g}}_{k},\qquad \text{ and } \qquad 
\langle f,g\rangle_{\mathbb{E}^{p}_{q}} =\sum_{k\in\mathbb{N}}
\mathrm{e}^{q k^p} \hat{f}_{k}\, \overline{\hat{g}_{k}} \, .
\end{equation*}
Then the corresponding norms induce the function spaces 
\begin{align*}
\mathbb{H}^s&=\{f\in L^2(\mathbb{R}): \|f\|_{\mathbb{H}^s} <\infty\},\quad s\geq 0,\\
\mathbb{E}^{p}_{q}& = \{f\in L^2(\mathbb{R}) :
                    \|f\|_{\mathbb{E}^{p}_{q}} <\infty\},\quad p,q>0
                    \, ,
\end{align*}
see also \cite{Thangavelu:1993aa}. The $\mathbb{H} ^s$ are the appropriate
versions of Sobolev spaces  for the specific basis $\{h_k\}_{k\in\N}$.

For the
  exponential weights $e^{qk^p}$, the $\bE ^p_q$ are related to
  Hilbert spaces of analytic  functions. We refer to both as Sobolev
  spaces in the following and to 
 Lemma~\ref{explainsob1} and \ref{explainsob2} for a more detailed explanation why
$\hs $  and $\mathbb{E}^p_q$ are  indeed  versions of  Sobolev spaces.

  For the study of the Sobolev spaces $\hs $ and $\bE ^p_q$ and for
  the error analysis of quadrature rules we need some detailed
  estimates on the associated (reproducing) kernels.\footnote{We write
  $\lesssim$ if the left-hand-side is bounded by a constant times the
  right-hand-side. If  $\lesssim$ and $\gtrsim$ both hold, then we write $\asymp$.}  
\begin{prop}\label{prop:new one again}
For $s>1-\frac{1}{\alpha}$ and $p,q>0$, we have 
\begin{align*}
\sum_{k=n}^\infty (1+k)^{-s}|h_{k}(x)|^2 &\lesssim (1+n)^{-s+1-\frac{1}{\alpha}},\\
\sum_{k=n}^\infty \mathrm{e}^{- q k^p} |h_{k}(x)|^2 &\lesssim \mathrm{e}^{-q n^p} (1+n)^{\frac{1}{3}-\frac{1}{\alpha}+\max(1-p,0)}.
\end{align*}
The constants depend on $s$, $\alpha$, and $p$, but they are independent
of $x$ and $n$. 
\end{prop}
\begin{proof}
We define the  Christoffel function associated to the orthonormal
basis $h_k$ by 
\begin{equation}\label{eq:christoffel}
\Lambda_{n}:=\frac{1}{\sum_{k=0}^n \left|h_{k}\right|^2}.
\end{equation}
Note that $\Lambda _n$ differs from the usual Christoffel function
$(\sum _{k=0}^n H_k^2)^{-1}$ by a factor $W^{-2}$.
The  precise growth of $\Lambda _n$  is
investigated  in
\cite{Levin:1992mb}. In particular, \cite[Eq.~(1.13)]{Levin:1992mb}
leads to the estimate
\begin{equation}\label{eq:bound for alpha and lambda-1}
\sum_{k=0}^n |h_{k}|^2 \lesssim n^{1-\frac{1}{\alpha}},\qquad n\geq 1.
\end{equation}

(i) Polynomial weights: Since $H_0$ is  constant and $h_{0}$ is bounded, without loss of generality, we may assume
$n\geq 1$. Choose
$M \in  \bN $, such that 
\begin{equation*}
2^M\leq n<2^{M+1} \, .
\end{equation*}
Then  dyadic summation and the bound for the Christoffel function~\eqref{eq:bound for alpha and lambda-1} yield
\begin{align*}
\sum_{k=n}^\infty (1+k)^{-s}|h_{k}(x)|^2 & \leq \sum_{m=M}^\infty \sum_{2^m\leq k<2^{m+1}}  2^{-sm}|h_{k}(x)|^2\\
& \leq  \sum_{m=M}^\infty   2^{-sm}2^{(m+1)(1-\frac{1}{\alpha})}\\
& \leq 2^{-sM} 2^{(M+1)(1-\frac{1}{\alpha})}\sum_{m=M}^\infty   2^{-s(m-M)}  2^{(m-M)(1-\frac{1}{\alpha})}\\
& \lesssim n^{-s+1-\frac{1}{\alpha}} \sum_{m=0}^\infty 2^{-m(s-1+\frac{1}{\alpha})}.
\end{align*}
Since $s>1-\frac{1}{\alpha}$, the series $\sum_{m=0}^\infty 2^{-m(s-1+\frac{1}{\alpha})}$ converges. 

(ii) Exponential weights: Here we use the following estimate for
the supremum norm of the orthogonal basis from~\cite[Eq.~(1.7)]{Levin:1992mb},
\begin{equation}\label{eq:bound uni}
\sup _{x\in \bR } |h_{k}(x)|^2\lesssim k^{\frac{1}{3}-\frac{1}{\alpha}},\qquad k\geq 1.
\end{equation}
This bound leads to
\begin{align*}
\sum_{k=n}^\infty \mathrm{e}^{-qk^p} |h_{k}(x)|^2 & \lesssim \sum_{k=n}^\infty \mathrm{e}^{-q k^p} k^{\frac{1}{3}-\frac{1}{\alpha}}.
\end{align*}
For $0<p<1$ and analogous to part (i), we partition the sum into blocks 
\begin{equation}\label{eq:partitioning}
m^{\frac{1}{p}}\leq k<(m+1)^{\frac{1}{p}}
\end{equation}
of size $(m+1)^{\frac{1}{p}}-m^{\frac{1}{p}}\lesssim m^{\frac{1}{p}-1}$ with $M^{\frac{1}{p}}\leq n<(M+1)^{\frac{1}{p}}$. Then the bound \eqref{eq:bound uni} leads to
\begin{align*}
\sum_{k=n}^\infty \mathrm{e}^{- q k^p} |h_{k}(x)|^2 & \lesssim
 \sum_{m=M}^\infty \sum_{m^{\frac{1}{p}}\leq k<(m+1)^{\frac{1}{p}}}\mathrm{e}^{- q m} |h_{k}(x)|^2\\
& \lesssim  \sum_{m=M}^\infty \mathrm{e}^{- q m} (m+1)^{\frac{\frac{1}{3}-\frac{1}{\alpha}}{p}} m^{\frac{1}{p}-1}\\
& \lesssim \mathrm{e}^{- q M} (M+1)^{\frac{\frac{1}{3}-\frac{1}{\alpha}}{p}+\frac{1}{p}-1} \sum_{m=M}^\infty \mathrm{e}^{-q(m-M)} \left(\frac{m+1}{M+1}\right)^{\frac{\frac{1}{3}-\frac{1}{\alpha}}{p}+\frac{1}{p}-1} \\
& \lesssim  \mathrm{e}^{- q M} (M+1)^{\frac{\frac{1}{3}-\frac{1}{\alpha}}{p}+\frac{1}{p}-1} \sum_{m=0}^\infty \mathrm{e}^{-qm} \left(1+\frac{m}{1+M}\right)^{\frac{\frac{1}{3}-\frac{1}{\alpha}}{p}+\frac{1}{p}-1}.
\intertext{Since $1+\frac{m}{1+M}\leq 1+m$, the final sum is bounded independently of $M$ and hence independently of $n$. According to $(M+1)\asymp n^p$, we obtain}
\sum_{k=n}^\infty \mathrm{e}^{- q k^p} |h_{k}(x)|^2& \lesssim  \mathrm{e}^{- q M} (M+1)^{\frac{\frac{1}{3}-\frac{1}{\alpha}}{p}+\frac{1}{p}-1} \lesssim  \mathrm{e}^{-q n^p} n^{\frac{1}{3}-\frac{1}{\alpha}+1-\frac{1}{p}}.
\end{align*}

For $p\geq 1$, the block size of the partitioning \eqref{eq:partitioning} satisfies $(m+1)^{\frac{1}{p}}-m^{\frac{1}{p}}\lesssim 1=m^{0}$. Then the above calculations lead to
\begin{equation*}
\sum_{k=n}^\infty \mathrm{e}^{- q k^p} |h_{k}(x)|^2 \lesssim  \mathrm{e}^{-q n^p} n^{\frac{1}{3}-\frac{1}{\alpha}},
\end{equation*}
which concludes the proof.
\end{proof}
In the case of $p<\frac{1}{3}$, it is beneficial to apply \eqref{eq:bound for alpha and lambda-1} instead of \eqref{eq:bound uni} in the above proof. Then the exponent $\frac{1}{3}-\frac{1}{\alpha}+1-p$ in Proposition \ref{prop:new one again} can even be replaced with $1-\frac{1}{\alpha}$. 

Using Proposition~\ref{prop:new one again}, we see easily that the
Sobolev spaces $\hs$ and $\bE ^p_q$ are reproducing kernel Hilbert
spaces.
\begin{lemma} \label{lem:rk}
(i) For  $s>1-\frac{1}{\alpha}$, $\hs $ is a  reproducing kernel
Hilbert space with  reproducing kernel
\begin{equation}
  \label{eq:cc1}
K_{\mathbb{H}^s}(x,y)  = \sum_{k\in\N}(1+k)^{-s}h_k(x)h_k(y)\nonumber
\, .
\end{equation}
(ii) For $p,q>0$, $\bE ^p_q$ is a  reproducing kernel
Hilbert space with  reproducing kernel
\begin{equation}
\label{eq:kernel Me} 
K_{\mathbb{E}^p_q}(x,y)  = \sum_{k\in\N}\mathrm{e}^{-q
  k^p}h_k(x)h_k(y).
\end{equation}
\end{lemma}
\begin{proof}
This follows immediately from the fact that the series defining
$K_{\hs } $ and $K_{\bE ^p_q}$ converge uniformly on $\bR \times \bR $
and that $(1+k)^{-s/2} h_k, k\in \bN ,$ is an orthonormal basis for $\hs $ with
respect to its inner product. Likewise $e^{-qk^p/2} h_k, k\in \bN ,$ is
an orthonormal basis for $\mathbb{E}^p_q$.    
\end{proof}

\section{Error bounds for Gauss-Hermite quadrature}\label{sec:GQ}
In this section, we derive error bounds for Gauss-Hermite quadrature
in $\mathbb{H}^s$ and $\mathbb{E}^{p}_{q}$. The  quadrature
nodes are the $n$ zeros of $h_{n}$, which we  denote by $X_n$, so that $\#
X_n = n$. The quadrature weights are the  Christoffel numbers  
\begin{equation}\label{eq:omega through Christoffel}
\omega(x)=\Lambda_{n}(x) W(x)= \big(\sum_{k=0}^n \left|h_{k}(x)\right|^2
\big)^{-1} W(x) ,\qquad x\in X_n \, ,
\end{equation}
cf.~\cite{Nevai:1986ak}.
These are positive and 
satisfy the exact quadrature relations
\begin{equation}\label{eq:quadrature}
\int_{-\infty}^\infty h_{k}(x)W(x)\mathrm{d}x = \sum_{x\in X_n} \omega(x) h_{k}(x),\qquad k\leq 2n-1.
\end{equation}
For a continuous function $f:\mathbb{R}\rightarrow\mathbb{R}$, the
Gauss-Hermite quadrature formula is then  
\begin{equation*}
Q_n(f)= \sum_{x\in X_n} \omega(x) f(x) \, ,
\end{equation*}
and $Q_n(f)$ provides an approximation of the actual integral
$\int_{-\infty}^\infty f(x)W(x)\mathrm{d}x$.

Note that we formulate the Gauss quadrature with respect to  the
orthonormal basis $\{h_k\}$  and the
measure $W(x) \, \mathrm{d}x$ instead of the  orthogonal polynomials
$H_k$ and the measure $W(x)^2\, \mathrm{d}x$. This requires some tiny
adaptation of the classical formulas. 

In the following we bound the error for functions in $\mathbb{H}^s$ and $\mathbb{E}^{p}_{q}$.
\begin{tm}\label{thm:alles}
For $1<\alpha\in 2\mathbb{N}$ with $s>1-\frac{1}{\alpha}$ and $p,q>0$, we have 
\begin{align}
\sup_{\substack{f\in \mathbb{H}^s\\ \|f\|_{\mathbb{H}^s}\leq
  1}}\left|\int_{-\infty}^\infty f(x)W(x)\mathrm{d}x-\sum_{x\in X_n}
  \omega(x)f(x)\right|^2   &\lesssim
                             n^{-s+1-\frac{1}{\alpha}} \label{eq:cc3} \\
\sup_{\substack{f\in \mathbb{E}^{p}_{q}\\ \|f\|_{\mathbb{E}^{p}_{q}}\leq 1}}\left|\int_{-\infty}^\infty f(x) W(x)\mathrm{d}x-\sum_{x\in X_n}\omega(x) f(x) \right|^2 &\lesssim \mathrm{e}^{-q(2n)^p}  n^{\frac{1}{3}-\frac{1}{\alpha}+\max(1-p,0)}
                          \notag 
\end{align}
\end{tm}
For the proof of Theorem \ref{thm:alles}, we require some bounds on the
sum of the Christoffel numbers. In the proof of \cite[Proposition
1]{Irrgeher:2015ab}, a suitable bound is derived for the case
$\alpha=2$, but the idea still works for all even integers $\alpha =
2k$, $k\in \bN $:
\begin{lemma}\label{lemma:sum christoffel}
If $1<\alpha\in2\mathbb{N}$, then we have
\begin{equation}\label{eq:bound on weights}
\sum_{x\in X_n} \omega(x) \leq 2. 
\end{equation}
\end{lemma}
\begin{proof}
According to \cite[Eq.~(8.4.6)]{Hildebrand:2003aa} (with weight
$W^2$ and $f = W^{-1}$), there are $\xi_{n}\in\mathbb{R}$ and $c_n\geq 0$, such that
\begin{equation*}
\int_{-\infty}^\infty W(x)\mathrm{d}x = \sum_{x\in X_n} \omega(x) + c_n \partial^{2n}W^{-1}(\xi_{n}).
\end{equation*}
For even integers $\alpha>1$, one may check by induction that 
\begin{equation*}
\partial^{2n} W^{-1} \geq 0.
\end{equation*}
Thus, we have
\begin{equation*}
 \sum_{x\in X_n} \omega(x) \leq \int_{-\infty}^\infty
 W(x)\mathrm{d}x=\int_{-\infty}^\infty e^{-\pi |x|^\alpha } \mathrm{d}x = 2\pi^{-\frac{1}{\alpha}} \Gamma(1+\tfrac{1}{\alpha})\leq 2,
\end{equation*}
which concludes the proof.
\end{proof}
We now derive  Theorem \ref{thm:alles}.
\begin{proof}[Proof of Theorem \ref{thm:alles}] 
Since $\mathbb{H}^s$ is a reproducing kernel Hilbert space, point
evaluation is continuous, and therefore 
\begin{equation*}
T:f\mapsto \int_{-\infty}^\infty f(x)W(x)\mathrm{d}x-\sum_{x\in X_n} \omega(x) f(x)
\end{equation*}
is a continuous linear functional  on $\mathbb{H}^s$. The
left-hand-side of the first inequality \eqref{eq:cc3} in Theorem
\ref{thm:alles} is  the square $\|T\|^2$ of the norm  of $T$.
By the Riesz representation theorem $T$ can be identified with a
function in $\hs $. Since  $\{(1+k)^{-\frac{s}{2}}
h_{k}\}_{k\in\mathbb{N}}$ is an orthonormal basis for $\mathbb{H}^s$,
 the norm of $T$ is given by 
\begin{align*}
\|T\|^2& = \sum_{k =0}^\infty |T\big((1+k)^{-\frac{s}{2}} h_{k}\big)|^2 \\
 &= \sum_{k=0}^\infty (1+k)^{-s} \left|\int_{-\infty}^\infty h_{k}(x) W(x)\mathrm{d}x-\sum_{x\in X_n} \omega(x) h_{k}(x)\right|^2.
\end{align*}
Since the quadrature rule is exact for $k\leq 2n-1$ by~\eqref{eq:quadrature}, the latter sum
starts at $k=2n$. Furthermore,  $h_0= c_0W$ is orthogonal  to
$h_k$ for $k\geq 1$, and  we obtain a closed form of the error: 
\begin{align*}\sum_{k =0}^\infty |T((1+k)^{-\frac{s}{2}} h_{k})|^2 & = \sum_{k=2n}^\infty (1+k)^{-s} \left|\sum_{x\in X_n} \omega(x) h_{k}(x)\right|^2\\
 & = \sum_{x,y\in X_n} \omega(x)\omega(y) \sum_{k=2n}^\infty
   (1+k)^{-s}h_{k}(x)h_{k}(y)\, .
   \end{align*}
An application of Cauchy-Schwartz and the kernel estimate of Proposition \ref{prop:new one again} lead to
\begin{align*}
\sum_{k=2n}^\infty (1+k)^{-s}h_{k}(x)h_{k}(y) & \leq \Big(\sum_{k=2n}^\infty (1+k)^{-s}|h_{k}(x)|^2  \sum_{k=2n}^\infty (1+k)^{-s}|h_{k}(y)|^2\Big)^{1/2}\\
&\leq (1+2n)^{-s+1-\frac{1}{\alpha}} \, .
\end{align*}
We are left with 
\begin{equation}\label{eq:1221}
\sum _{x,y \in X_n} \omega (x) \omega (y) \leq 4
\end{equation}
by
Lemma \ref{lemma:sum christoffel}. Consequently, $\|T\|^2 \lesssim
n^{-s +1-\frac{1}{\alpha}}$,  which  concludes the proof of
\eqref{eq:cc3}  for $\mathbb{H}^s$.

The proof for $\mathbb{E}^{p}_{q}$ is derived analogously. We note
that $e^{-qk^p/2}h_k, k\in \bN $, is an orthonormal basis for
$\mathbb{E}^p_q$ and find as above that $T$ on $\mathbb{E}^p_q$ satisfies
\begin{align}
\|T\|^2 &=\sup_{\substack{f\in \mathbb{E}^{p}_{q} \\ \|f\|_{\mathbb{E}^{p}_{q}}\leq
    1}}\left|\int_{-\infty}^\infty f(x) W(x)\mathrm{d}x-\sum_{x\in
    X_n}\omega(x) f(x) \right|^2 \\
    &=  \sum_{x,y\in X_n}
\omega(x)\omega(y) \sum_{k=2n}^\infty e^{qk^p}\,   h_{k}(x)h_{k}(y)\, .  \label{eq:ch2}
\end{align}
Again, \eqref{eq:1221} with Proposition \ref{prop:new one again} concludes the proof.
\end{proof}
\begin{rem}\label{rem:irre}
Let us now consider $\alpha=2$ only. We expect that the optimal
exponent for Gaussian quadrature in $\mathbb{H}^s$ is ${-s}$ and the
additional $\frac{1}{2}$ is an artifact of our proof technique,
cf.~\cite{Dick:2018aa,Kazashi:2022mn}. The case $\mathbb{E}^{p}_q$
with $p\geq 1$ is also covered by results in \cite{Irrgeher:2015ab},
whereas we have the additional factor $n^{-\frac{1}{6}}$. While
\cite{Irrgeher:2015ab} addresses multivariate functions in a rather
general setting, the case  $0<p<1$ is excluded there. 
\end{rem}

\section{Modulation spaces}\label{sec:modulation}
In this  section we restrict ourselves to the exponent $\alpha = 2$
and  investigate the Sobolev-type spaces associated to
the Gaussian weight $W(x) = e^{-\pi x^2}$. Our main insight is the
identification of the spaces $\hs $ with a class of well-known
function spaces from analysis that have been used in \tfa ,  for the
analysis of pseudodifferential operators, for the
description of uncertainty principles, etc. See~\cite{BO20,CR20,gabbook1,gabbook2,book} for an
exposition of their theory and their relevance. Certain \modsp s have
been introduced in~\cite{Dick:2018aa,Irrgeher:2015aa,Irrgeher:2015ab,IKPW16} under the name Hermite spaces.    

For $\alpha =2$,   $h_0$ is  the normalized Gaussian 
\begin{equation*}
\varphi(x)=2^{\frac{1}{4}}e^{-\pi x^2}, 
\end{equation*} 
and the associated orthogonal polynomials are the Hermite polynomials,
the associated orthonormal basis for $L^2(\bR )$ is the Hermite
basis $\{h_k\}$.

We  first define a special family of \modsp s. These are defined
 by imposing a norm on the short-time Fourier transform. Recall that the
 \stft\  of $f\in L^2(\mathbb{R})$ (with respect to the Gaussian $\varphi
 $) is given
by 
\begin{equation*}
V_{\varphi} f(x,\xi) = \int_{-\infty}^\infty f(t) \varphi(t-x) e^{- 2 \pi \mathrm{i} \xi t}
dt ,\qquad x,\xi\in\mathbb{R}.
\end{equation*}
We apply the standard identification $(x,\xi)\simeq
z=x+i\xi\in\mathbb{C}$.
Then a weighted $L^2$-norm with respect to a polynomial or an
exponential weight function leads to the  norms
\begin{align*}
\|f\|^2_{\mathbb{M}^s} &= \int _{\bR ^2} (1+|z|^2) ^{s}|V_{\varphi} f(z)|^2  \mathrm{d}z,\\
\|f\|^2_{\mathbb{M}^s_{\mathrm{e}}} &= \int _{\bR ^{2}}\mathrm{e}^{s|z|} |V_{\varphi} f(z)|^2 \mathrm{d}z,\\
\|f\|^2_{\mathbb{M}^s_{\mathrm{e}^2}} &= \int _{\bR ^{2}} \mathrm{e}^{s|z|^2}|V_{\varphi}
     f(z)|^2 \mathrm{d}z  \, ,
\end{align*}
and  the modulation spaces
\begin{align*}
\mathbb{M}^s &=\{f\in L^2(\mathbb{R}) : \|f\|_{\mathbb{M}^s}<\infty\},\quad s\geq 0,\\
\mathbb{M}^s_{\mathrm{e}} &=\{f\in L^2(\mathbb{R}) : \|f\|_{\mathbb{M}^s_{\mathrm{e}}}<\infty\},\quad s\geq 0,\\
\mathbb{M}^s_{\mathrm{e}^2} &=\{f\in L^2(\mathbb{R}) : \|f\|_{\mathbb{M}^s_{\mathrm{e}^2}}<\infty\}, \quad \pi>s\geq 0.
\end{align*}

The following theorem identifies the Sobolev spaces $\hs $ and
$\bE ^p_q$ as \modsp s with respect to a polynomial or exponential
weight\footnote{In the standard notation $\mathbb{M}^s$ is the
  modulation space $M^{2,2}_{v_s}$ with $v_s(z)=(1+|z|^2)^s$ and $\mathbb{M}^s_{\mathrm{e}}$ is
  $M^{2,2} _{w_s}$ with $w_s(z) = e^{s|z|}$. Since we use only
  modulation spaces that are Hilbert spaces, we have simplified the notation.}. 
We feel that the existing results for \modsp s could be
useful for numerical analysts. 
\begin{tm}\label{thm:poly characterized}
For $\alpha=2$, the following identities hold with equivalent norms: 
\begin{itemize}
\item[(i)] $\mathbb{H}^{s} = \mathbb{M}^s$, for $s\geq 0$,
\item[(ii)] $\mathbb{E}^{\frac{1}{2}}_q  = \mathbb{M}^s_{\mathrm{e}}$, for $q=\frac{s}{\sqrt{\pi}}$ and $s\geq 0$,
\item[(iii)] $\mathbb{E}^{1}_q = \mathbb{M}^s_{\mathrm{e}^2}$, for $q=\ln(\frac{\pi}{\pi-s})$ and $\pi>s\geq 0$.
\end{itemize}
\end{tm}
\begin{proof}
We use the Bargmann transform to translate the norms to an equivalent
norm for a subspace of  the Bargmann-Fock space
$\mathcal{F}$~\cite{Folland:1989aa}. This space  consists of all entire functions $F$ such that 
\begin{equation*}
\|F\|_{\mathcal{F}}^2 = \int_{\mathbb{C}} \left| F(z) \right|^2 \mathrm{e}^{-\pi |z|^2}\mathrm{d}z
\end{equation*}
is finite. The normalized monomials $\{\sqrt{ \frac{\pi ^k}{k!}}
  \,z^k\}_{k \in \bN }$ form an orthonormal basis for $\mathcal{F}$ with respect to the inner product 
 \begin{equation*}
 \langle F,G\rangle_{\mathcal{F}} =\int_{\mathbb{C}} F(z) \overline{G(z)} \mathrm{e}^{-\pi |z|^2}\mathrm{d}z.
 \end{equation*}
The Bargmann transform defined by 
\begin{equation*}
Bf(z) = 2^{1/4} e^{-\pi z^2/2}\,  \int _{-\infty }
^\infty f(x) e^{2\pi z x} e^{-\pi x^2} \, \mathrm{d}x
\end{equation*}
 is a unitary operator
from $L^2(\bR )$ onto $\cF $ and maps the Hermite functions to the
normalized monomials, i.e., 
\begin{equation}\label{eq:Bargman}
\mathcal{B}: L^2(\mathbb{R})\longrightarrow\mathcal{F},\qquad 
 h_k \mapsto \sqrt{ \frac{\pi ^k}{k!}} \,z^k.
\end{equation}
With the identification $(x,\xi ) \simeq z= x+i\xi \in \bC $, the
connection of $B$ to the \stft\ is given by the formula  
\begin{equation}\label{eq:V and B}
V_\varphi f (z)  = \mathrm{e}^{-\mathrm{i} x \xi} (Bf)(\bar z) e^{-\frac{\pi}{2} |z|^2},
\end{equation}
cf.~\cite{Folland:1989aa,Grochenig:2011aa}. For the proof  we  use
the Hermite  expansion
\begin{equation*}
f=\sum_{k=0}^\infty \hat{f}_k h_k 
\end{equation*}
of  $f\in L^2(\mathbb{R})$. The main point of the proof is the fact
that the monomials $z^k, k\in \bN $, are still orthogonal with respect to 
the weighted inner product $\int F(z) \overline{G(z)} w(z) e^{-\pi
  |z|^2} \, \mathrm{d}z$ for an arbitrary rotation-invariant weight
function $w$. 

(i) We first  consider $\mathbb{M}^s$ and obtain
\begin{align*}
  \|f\|_{\mathbb{M}^s} ^2 &= \int_{\bR ^2} \left|\sum_{k=0}^\infty \hat{f}_k V_\varphi h_k (z)\right|^2 (1+|z|^2)^s \, dz \\
  & = \int_{\bR ^2} \left|\sum_{k=0}^\infty \hat{f}_k  \mathrm{e}^{-\mathrm{i} x \xi} (Bh_k)(\bar z)  \right|^2 (1+|z|^2)^s e^{-\pi |z|^2}\, dz \\
    & = \int_{\bR ^2} \left|\sum_{k=0}^\infty \hat{f}_k  \sqrt{ \frac{\pi ^k}{k!}} \,\bar{z}^k  \right|^2 (1+|z|^2)^s e^{-\pi |z|^2}\, dz \\
&= \sum _{k,l} \hat{f}_k \overline{\hat{f}_l} \sqrt{\frac{\pi ^k}{k!}}
\sqrt{\frac{\pi ^l}{l!}} \int_{\bR ^2} \bar{z}^k z^l   (1+|z|^2)^s
  e^{-\pi |z|^2} \, dz .
  \intertext{Since $ (1+|z|^2)^s$ is rotationally invariant, the monomials are still orthogonal with respect to $(1+|z|^2)^s
  e^{-\pi |z|^2} \mathrm{d}z$ and we obtain}
 \|f\|_{\mathbb{M}^s} ^2&= \sum _{k=0}^\infty |\hat{f}_k|^2 \frac{\pi ^k}{k!} \int _{\bR ^2} 
|z|^{2k} (1+|z|^2)^s e^{-\pi |z|^2} \, dz .
\end{align*}
The use of polar coordinates and the  substitution $t=\pi r^2$ lead to
\begin{align*}
\frac{\pi ^k}{k!} \int _{\bR ^2} 
|z|^{2k} (1+|z|^2)^s e^{-\pi |z|^2} \, dz  & = \frac{2\pi ^{k+1}}{k!} \int_0^\infty r^{2k+1}(1+r^2)^s e^{-\pi r^2}\mathrm{d}r\\
& = \frac{1}{k!} \int_0^\infty t^{k}\big(1+\frac{t}{\pi}\big)^s e^{- t}\mathrm{d}t\\
& \asymp \frac{1}{k!} \int_0^\infty t^{k}\big(1+\frac{t^s}{\pi^s}\big)  e^{- t}\mathrm{d}t \\
&= \frac{1}{\Gamma (k+1)} \Big(  \Gamma (k+1) + \frac{\Gamma
  (k+s+1)}{\pi ^s} \Big) \, .
\end{align*}
The standard asymptotics
$$
\frac{\Gamma (k+s+1)}{\Gamma (k+1)} \asymp 
k^s,
$$
cf.~\cite[5.11.12]{dlmf}, 
lead to \begin{equation*}
\frac{\pi ^{k+1}}{k!} \int_0^\infty t^{k} (1+t^s) e^{-\pi
  t}\mathrm{d}t \asymp (1+k^s)\asymp (1+k)^s\, .
\end{equation*}
Consequently,
$$
\|f\|_{\mathbb{M} ^s}^2 \asymp \sum _{k=0}^\infty (1+k)^s |\hat{f}_k|^2 =
\|f\|_{\mathbb{H}^s}\, ,
$$
which concludes the proof of the identification of  $\mathbb{M}^s$
with $\mathbb{H}^s$. 

(iii) Because it is easier, we prove (iii) next. As  before
\begin{align*}
  \|f\|_{\mathbb{M}^s_{\mathrm{e}^2}} ^2 &= \int_{\bR ^2} |V_\varphi f(z)|^2 e^{s |z|^2}\, dz \\
&= \sum _{k=0}^\infty |\hat{f}_k|^2 \frac{\pi ^k}{k!} \int _{\bR ^2} 
|z|^{2k}  e^{-(\pi-s) |z|^2} \, dz .
\end{align*}
Polar coordinates and the substitution $u=(\pi-s)r^2$ lead to
\begin{align*}
  \frac{\pi ^k}{k!} \int _{\bR ^2} |z|^{2k} e^{-(\pi-s) |z|^2} \, dz&=
    \frac{\pi ^k}{k!} \int _0^\infty  r^{2k} e^{-(\pi-s)r^2} \, 2\pi  r dr\\
    & = \frac{\pi^{k+1}}{k! (\pi-s)^{k+1}}\int _0^\infty  u^{k} e^{-u}   du\\
    & =  \left(\frac{\pi}{\pi-s}\right)^{k+1}\, .
\end{align*}
Consequently,
\begin{align*}
  \|f\|_{\mathbb{M}^s_{\mathrm{e}^2}} = \sum _{k=0}^\infty
  |\hat{f}_k|^2 \Big(\frac{\pi}{\pi - s}\Big)^{k+1} = \frac{\pi}{\pi -
  s} \sum _{k=0}^\infty |\hat{f}_k|^2 \exp \Big( \ln (\frac{\pi}{\pi -
  s} ) k \Big) = \frac{\pi}{\pi - s} \|f\|_{\mathbb{E}^1_q} \, . 
\end{align*}
This
concludes the proof of the identity
$\mathbb{M}^s_{\mathrm{e}^2}=\mathbb{E}^1_q$ with $q = \ln (\tfrac{\pi
}{\pi -s})$.  

(ii) The proof for $\mathbb{M}^s_{\mathrm{e}}$ is more involved and
requires detailed formulas and estimates for hypergeometric functions.
We cite freely the digital library of mathematical functions DLMF~\cite{dlmf}. As above, we obtain 
\begin{align*}
  \|f\|_{\mathbb{M}^s_{\mathrm{e}}} ^2 &= \int_{\bR ^2} |V_\varphi f(z)|^2 e^{s|z|}\, dz \\
&= \sum _{k=0}^\infty |\hat{f}_k|^2 \frac{\pi ^k}{k!} \int _{\bR ^2} 
|z|^{2k} e^{s|z|} e^{-\pi |z|^2} \, dz .
\end{align*}
The use of polar coordinates and the identity $\cosh(x)+\sinh(x)=\mathrm{e}^x$ lead to
\begin{align*}
  \frac{\pi ^k}{k!} \int _{\bR ^2} &|z|^{2k} e^{s|z|} e^{-\pi |z|^2} \, dz=
    \frac{2\pi ^{k+1}}{k!} \int _0^\infty  r^{2k+1} e^{s r-\pi r^2} \, dr\\
    & = \frac{2\pi^{k+1}}{k!} \int_0^\infty r^{2k+1}e^{-\pi r^2} \cosh(sr)\mathrm{d}r 
    + \frac{2\pi^{k+1}}{k!} \int_0^\infty r^{2k+1}e^{-\pi r^2} \sinh(sr)\mathrm{d}r.
    \end{align*}
   For $p,q\in\mathbb{N}$, denote ${}_pF_q$ the hypergeometric
   function and observe 
     \begin{align*}
       \cosh(sr)&={}_0F_1(\frac{1}{2},(\frac{sr}{2})^2),\\
      \sinh(sr)&= {}_0F_1(\frac{3}{2},(\frac{sr}{2})^2)sr. 
        \end{align*}
    According to DLMF \cite[16.5.3]{dlmf}, for $a\in\mathbb{C}$ with positive real part, we have 
    \begin{align*}
{}_1F_1(a,b,z)  &= \frac{1}{\Gamma(a)} \int_0^\infty t^{a-1}e^{-t} {}_0F_1(b,zt)\mathrm{d}t,\\
& = \frac{2\pi^a}{\Gamma(a)} \int_0^\infty r^{2a-1}e^{-\pi r^2} {}_0F_1(b,z\pi r^2)\mathrm{d}r,
    \end{align*}
    where the second equality is due to the substitution $r = \sqrt{\frac{t}{\pi}}$. This leads to 
    \begin{align*}
    \frac{2\pi^{k+1}}{k!} \int_0^\infty r^{2k+1}e^{-\pi r^2} \cosh(sr)\mathrm{d}r &= {}_1F_1(k+1;\frac{1}{2},\frac{s^2}{4\pi})\\
    \frac{2\pi^{k+1}}{k!} \int_0^\infty r^{2k+1}e^{-\pi r^2} \sinh(sr)\mathrm{d}r & = \frac{s 2\pi^{k+1}}{k!} 
 \int_0^\infty r^{2k+2}e^{-\pi r^2}{}_0F_1(\frac{3}{2},(\frac{sr}{2})^2)\\
 & = \frac{q}{\pi^{\frac{1}{2}}} \frac{\Gamma(k+\frac{3}{2})}{\Gamma(k+1)}{}_1F_1(k+\frac{3}{2};\frac{3}{2};\frac{s^2}{4\pi}).
    \end{align*}
For large $a$, 
\begin{equation*}
{}_1F_1(a,b,z^2) \asymp \left(4a-2b\right)^{\frac{1-b}{2}}I_{b-1}\left(z\sqrt{4a-2b}\right),
\end{equation*}
cf.~\cite[Eq.~(4.6.42) in Section 4.6.1]{Slater:60}, where $I_{b-1}$ denotes the modified Bessel function. (Note that there is a square missing in \cite[Eq.~(2.2)]{Temme:2013aa}). According to DLMF~\cite[10.40.1]{dlmf} we have, for large $x$, 
\begin{equation*}
I_\nu(x) \asymp \frac{e^x}{\sqrt{2\pi x}}.
\end{equation*}
Here, we only consider $b=\frac{1}{2}$ or $b=\frac{3}{2}$, so that we need these asymptotics for $\nu=\pm \frac{1}{2}$, in which case we even have
\begin{equation*}
I_{\pm \frac{1}{2}}(x) = \frac{e^x \mp e^{-x}}{\sqrt{2\pi x}}.
\end{equation*}
Therefore, we obtain
\begin{equation*}
{}_1F_1(k+1,\frac{1}{2},\frac{s^2}{4\pi})\asymp (4k+3)^{\frac{1}{4}} \frac{e^{\frac{s}{2\sqrt{\pi}}\sqrt{4k+3}}}{\sqrt{2\pi \frac{s}{2\sqrt{\pi}}\sqrt{4k+3}}}\asymp e^{\frac{s}{\sqrt{\pi}}\sqrt{k}}.
\end{equation*}
Since $\frac{\Gamma(k+\frac{3}{2})}{\Gamma(k+1)}\asymp  \sqrt{k}$, we have
\begin{equation*}
\frac{\Gamma(k+\frac{3}{2})}{\Gamma(k+1)}{}_1F_1(k+\frac{3}{2};\frac{3}{2};\frac{s^2}{4\pi})  \asymp k^{\frac{1}{2}} (4k+3)^{-\frac{1}{4}}\frac{e^{\frac{s}{2\sqrt{\pi}}\sqrt{4k+3}}}{\sqrt{2\pi \frac{s}{2\sqrt{\pi}}\sqrt{4k+3}}}
 \asymp e^{\frac{s}{\sqrt{\pi}}\sqrt{k}},
\end{equation*}
which concludes the proof for $\mathbb{M}^s_{\mathrm{e}}$. 
\end{proof}

The identification $\mathbb{M}^s=\mathbb{H}^s$ in the case
$\alpha=2$  is also mentioned  in \cite[Eq.~(1.5)]{Janssen:2005aa}
without proof. 
 We also point out that  
our quadrature bounds for $\mathbb{M}^s_{\mathrm{e}}$ in Theorem
\ref{thm:alles} are not covered by \cite{Irrgeher:2015ab}, see also
Remark \ref{rem:irre}.

To complete the  discussion of modulation spaces (or Hermite spaces),
we add two relevant characterizations that explain why we consider
them as Sobolev spaces where  Sobolev space is understood as the domain
of a differential operator.

Consider the differential operator $\mathcal{L}f(x) = -\frac{1}{2\pi}
f''(x) + 2\pi x^2f(x)$. In quantum mechanics $\mathcal{L}$ is the Schr\"odinger
operator of the harmonic oscillator. Its eigenfunctions are precisely
the Hermite functions, and we have  $$
\mathcal{L}h_n = (2n+1)h_n
$$
in our normalization of the Hermite functions, see
\cite{Folland:1989aa}. Then we have the following identification.
\begin{lemma} \label{explainsob1}
  $\hs = \mathrm{dom}\, _{L^2} \mathcal{L}^{s/2}  = \mathcal{L}^{-s/2}
  \big( L^2(\bR )\big)$ with equivalence of norms. 
\end{lemma}
\begin{proof}
  Since $\mathcal{L}$ is diagonalized by the Hermite basis, it follows
  that
  $\mathcal{L}^{s/2}f\in L^2(\bR )$ \fif\ $\sum _{k=0}^\infty |\langle
  f, h_k\rangle |^2 (1+2k)^s <\infty $, and thus $\|\mathcal{L}^{s/2}f\|_2
  \asymp \|f\|_{\hs }$. 
\end{proof}
In this context $\hs $ is referred to as the Shubin class, and as
always $\hs $ coincides with the domain of much more general
pseudodifferential operators. See Shubin's
book~\cite[Ch.~25.3]{shubin} for a detailed exposition.

Next, for $s\geq 0$ let $L^2_s$ be the weighted $L^2$-space defined by
the norm $\|f\|_{L^2_s}^2 = \int_{\mathbb{R}} |f(x)|^2 (1+|x|)^{2s} \,
\mathrm{d}x$ and let $\mathcal{F} L^2_s = \{ f\in L^2: \hat{f} \in L^2_s\}$ be its image under the
Fourier transform. Note that $\mathcal{F} L^2_s$ is the standard 
Sobolev space on $\bR $. The following identification is mentioned 
in~\cite{BCG04,CPRT05,heil03}. 
\begin{lemma} \label{explainsob2}
  $\mathbb{M}^s = \hs = L^2_s(\bR ) \cap \mathcal{F} L^2_s(\bR )$. 
\end{lemma}

Finally we mention that the \modsp s $\hs = \mathcal{M}^s$ and
$\mathbb{E}^{\frac{1}{2}}_q=\mathbb{M}^s_{\mathrm{e}}$ are invariant
with respect to  time-frequency shifts defined by $f_{y,\eta } (x) =
e^{2\pi i \eta x} f(x-y)$ (shift by $y$ and modulation by $\eta$): if
$f\in \hs $, then $f_{y,\eta } \in \hs $ for all $y,\eta \in \bR $ and
similarly for $ \mathbb{M}^s_{\mathrm{e}}$. This is easy to see from 
the definition of $\mathbb{M}^s$, but is not at all obvious when using
the $\hs $-norm. This invariance property is just one example of how time-frequency
methods facilitate the investigation of general Hermite spaces.

\section{Numerical experiments}\label{sec:nums}
In this section we make the error estimates of Theorem~\ref{thm:alles}
more explicit for the \modsp s and compare them with numerical
simulations.  We provide some numerical experiments for Gaussian
quadrature nodes $X_n$ with associated Christoffel weights
$\{\omega(x)\}_{x\in X_n}$ in the modulation spaces
$\mathbb{M}^s_{\mathrm{e}^2}$ and $M^s_{\mathrm{e}}$.  
We omit the case $\mathbb{M}^s$ with polynomial weights, as it was 
already  presented in 
\cite{Dick:2018aa}. Note that the results in\cite{Dick:2018aa}
suggest a decay rate of $\cO (n^{-s})$ instead of the theoretical
bound $\cO (n^{-s+\frac{1}{2}})$ proved  in Theorem \ref{thm:alles}. 
\subsection{The modulation space $\mathbb{M}^s_{\mathrm{e}^2}$}\label{sec:num for Me2}
We choose $t>1$ and consider the (slightly modified) Mehler kernel
\begin{align} \label{eq:mehler}
K_t(x,y) &= \sum_{k=0}^\infty t^{-(k+1)} h_k(x)h_k(y) 
 = \sqrt{\frac{2}{t^2-1}} e^{\frac{\pi}{t^2-1}\left(4t xy - (t^2+1)( x^2+y^2)\right) }.
\end{align}
In view of the identification $\mathbb{M}^s_{\mathrm{e}^2}= \mathbb{E}^1_q$, \eqref{eq:kernel Me} in Lemma~\ref{lem:rk}  with
$t=\frac{\pi}{\pi-s}$ shows that  
$K_t$ is the reproducing kernel of $\mathbb{M}^s_{\mathrm{e}^2}$.

We now use a 
standard identity for the worst case integration error in terms of the
reproducing kernel, see, e.g., \cite{Graf:2013zl} or
\cite[Prop.~2.11]{DP10a}. We have 
\begin{align}
\lefteqn{\sup_{\substack{f\in \mathbb{M}^s_{\mathrm{e}^2},
  \|f\|_{\mathbb{M}^s_{\mathrm{e}^2}}\leq 1}} \left|
  \int\limits_{\,-\infty}^\infty f(x) W(x) \mathrm{d}x- \sum_{x\in
  X_n}\omega(x) f(x)\right|^2}\notag \\
  & =  \int\limits_{-\infty}^\infty \int\limits_{-\infty}^\infty
    K_t(x,y)W(x)\mathrm{d}x W(y)\mathrm{d}y + \sum_{x,y\in
    X_n}\omega(x) \omega(y)  K_t(x,y) \notag \\
& \qquad \qquad  -2\sum_{x\in X_n} \omega(x)   \int\limits_{-\infty}^\infty K_t(x,y)W(y)\mathrm{d}y,
\label{ju2} \end{align}
For the Mehler kernel \eqref{eq:mehler} we can make the worst case error
more explicit. Since $W= 2^{-1/4} h_0  $ and 
\begin{equation*}
\int_{-\infty}^{\infty} W^2(x)\mathrm{d}x = \int _{-\infty }^\infty
e^{-2\pi x^2} \, \mathrm{d}x=  \frac{1}{\sqrt{2}}, 
\end{equation*}
 the orthogonality of the Hermite functions   implies 
\begin{align*}
  \int_{-\infty}^\infty K_{t}(x,y)W(x)\mathrm{d}x
  &=   \sum_{k=0}^\infty t^{-k-1} \int   _{\infty } ^\infty  h_k(x) \,
    W(x) dx \, h_k(y) \\
&=   \sum_{k=0}^\infty t^{-k-1} 2^{-1/4} \langle h_k,h_0 \rangle
                           h_k(y) =  t^{-1} W(y)\, .
\end{align*}
A further integration with respect to $y$ yields 
\begin{equation*}
\int\limits_{-\infty}^{\infty}\int\limits_{-\infty}^{\infty}
K_{t}(x,y)W(x)\mathrm{d}x \, W(y)\mathrm{d}y = t\inv \int _{-\infty
}^\infty W(y)^2 \, dy =  \frac{1}{\sqrt{2}t }.
\end{equation*}
For the third term in \eqref{ju2} we use the exact quadrature formula
\eqref{eq:quadrature} for $k=0$ and obtain that 
\begin{align*}
-2\sum_{x\in X_n} \omega(x)   \int_{-\infty}^\infty K_t(x,y)W(y)\mathrm{d}y & =  -2t^{-1}\sum_{x\in X_n}\omega(x) W(x)\\
& = -2 t^{-1} \int_{-\infty}^\infty W^2(x)dx=-\frac{2}{\sqrt{2} t}.
\end{align*}
Consequently, we now have 
\begin{align}
\mathrm{WCE}(n, \mathbb{M}^2_{\mathrm{e}^2})& :=  \sup_{\substack{f\in \mathbb{M}^s_{\mathrm{e}^2}, 
    \|f\|_{\mathbb{M}^s_{\mathrm{e}^2}}\leq 1}}\left|
  \int\limits_{\,-\infty}^\infty f(x) W(x) \mathrm{d}x- \sum_{x\in
     X_n} \omega(x)f(x)\right|^2 \notag \\
  &= \sum_{x,y\in X_n}\omega(x)
\omega(y)K_{t}(x,y)-\frac{1}{\sqrt{2} \, t}. \label{eq:kernel wce}
\end{align}
By Theorem \ref{thm:poly characterized}(iii)  the \modsp\ $\mathbb{M}^s_{\mathrm{e}^2}$
coincides with the Sobolev-type  space $\mathbb{E}^1_q$ with parameter $q = \ln
\frac{\pi}{\pi -s} = \ln t $. For the case  $p=1$    the theoretical
predication of    
Theorem \ref{thm:alles}   ensures that the decay rate of the worst case
error \eqref{eq:kernel wce} is of the order 
$$ \mathrm{WCE}(n, \mathbb{M}^2_{\mathrm{e}^2}) = \cO \left(e^{-q(2n)^p}\right) = \cO \left( (e^{-q})^{2n}\right) = \cO \left(
t^{-2n} \right) \, ,
$$
so that $\log _{10} \mathrm{WCE}(n, \mathbb{M}^2_{\mathrm{e}^2}) \leq
C - 2n\log_{10}(t)$. This means that  the logarithm of the worst case quadrature error from~\eqref{eq:kernel wce} 
must stay below a line of slope
$-2\log_{10}(t)$.

In the  numerical experiments we have evaluated  the explicit
formula~\eqref{eq:kernel wce} for various values of $t$, namely   
$t_1=\tfrac{\pi}{\pi -s_1}=\frac{5}{4}$ and $t_2=\tfrac{\pi}{\pi -s_1}=\frac{50}{49}$. Note that the closed
form~\eqref{eq:mehler}  of the Mehler kernel  $K_t$  permits a direct evaluation of $
\mathrm{WCE}(n, \mathbb{M}^2_{\mathrm{e}^2})$.  

We then let $n$ run through the odd integers from $3$ to $41$ and plot
the resulting line,  see Figure \ref{fig:mehler}. We  observe that the actual slope is
slightly steeper than the theoretical prediction.
\begin{figure}
\subfigure[$t_1=\frac{5}{4}$, slope $-0.35$]{
\includegraphics[width=.45\textwidth]{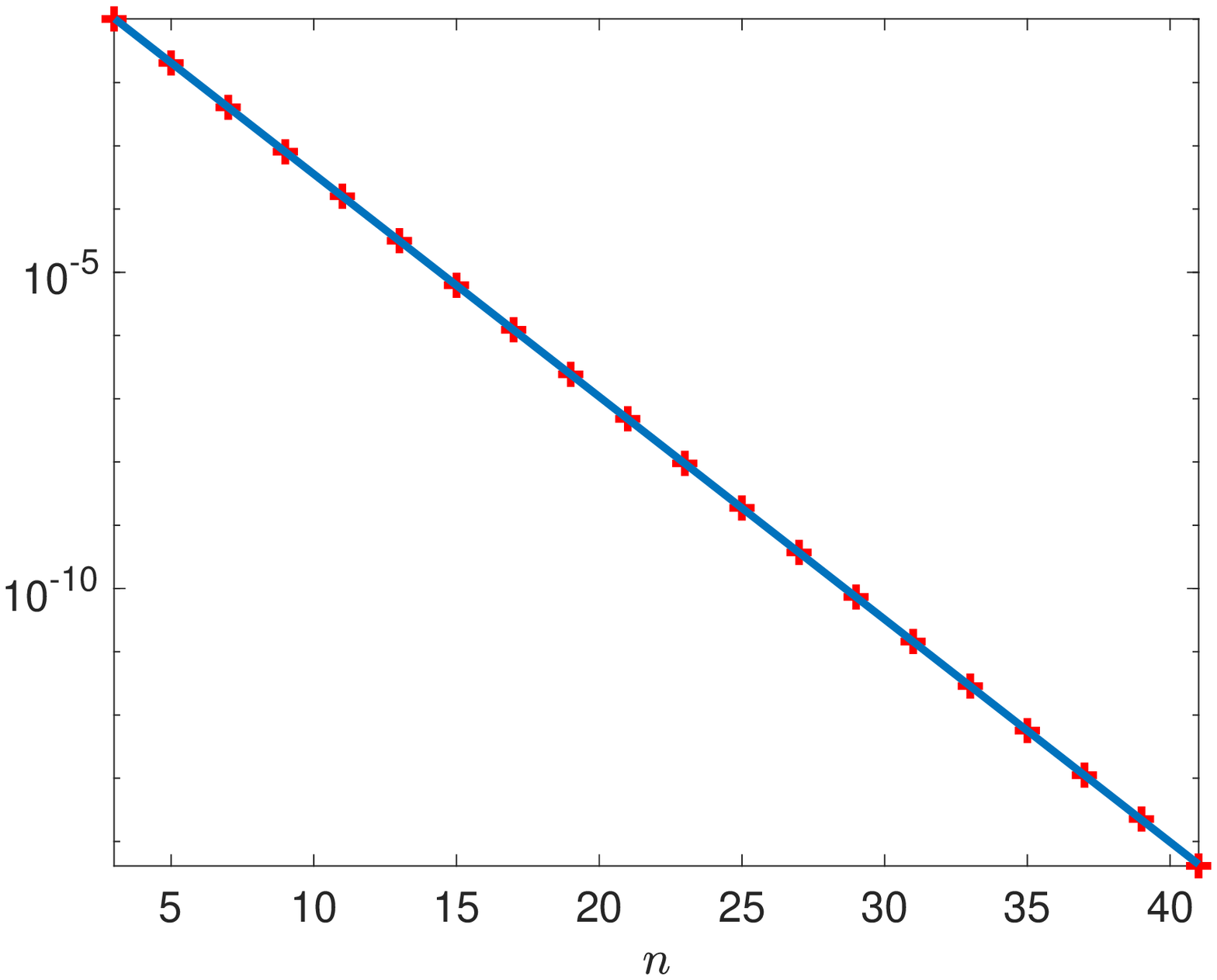}}
\subfigure[$t_2=\frac{50}{49}$, slope $-0.03$]{
\includegraphics[width=.45\textwidth]{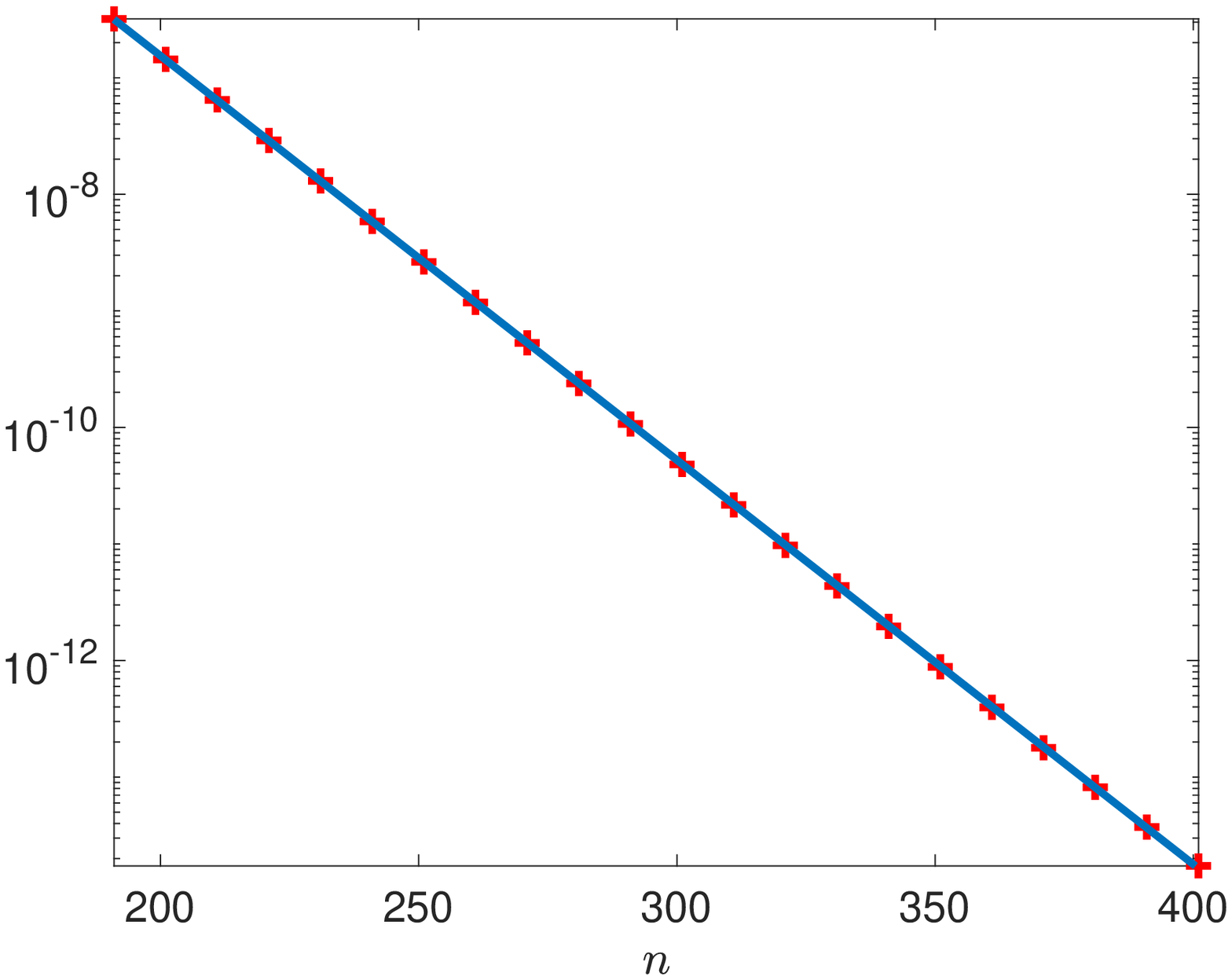}}
\caption{Logarithmic plot of the Gaussian quadrature error in
  $\mathbb{M}^s_{\mathrm{e}^2}$ against $n$ for
  $s_i=(1-\frac{1}{t_i})\pi$ with $t_1=\frac{5}{4}$ and
  $t_2=\frac{50}{49}$, for $i=1,2$. The blue line is a linear least
  squares fit of the actual  evaluation of the   worst case
  quadrature error \eqref{eq:kernel wce} at odd $n$ marked in
  red. The
 slope of the least squares fit is slightly steeper than the
 theoretical slope  $-2\log_{10}(t_1)\approx -0.20$ and
 $-2\log_{10}(t_2)\approx -0.02$. 
}\label{fig:mehler}
\end{figure}

\subsection{The modulation space $\mathbb{M}^s_{\mathrm{e}}$}\label{sec:Me1}
We use the equivalence of norms in
$\mathbb{M}^s_{\mathrm{e}}=\mathbb{E}^{\frac{1}{2}}_q$ with
$q=\frac{s}{\sqrt{\pi}}$ proved in   Theorem \ref{thm:poly
  characterized}(ii). By   the proof of Theorem
\ref{thm:alles}, in particular~\eqref{eq:ch2},  the worst case
quadrature error in $\mathbb{M}^s_{\mathrm{e}}$ is given by  
\begin{align}
\mathrm{WCE}(n, \mathbb{M}^2_{\mathrm{e}})&:= \sup_{\substack{f\in \mathbb{M}^s_{\mathrm{e}}, 
  \|f\|_{\mathbb{M}^s_{\mathrm{e}}}\leq 1}} \left|
  \int\limits_{\,-\infty}^\infty f(x) W(x) \mathrm{d}x- \sum_{x\in  X_n}\omega(x)
   f(x)\right|^2 \notag \\
&  \asymp \sum_{x,y\in X_n}
  \omega(x)\omega(y) \sum_{k=2n}^\infty
  \mathrm{e}^{-q\sqrt{k}}h_{k}(x)h_{k}(y). \label{eq:remainder of wce}
\end{align}
Theorem \ref{thm:alles} ensures that \eqref{eq:remainder of wce}
decays at least as $\mathrm{e}^{-q\sqrt{2n}}$.  Thus 
\begin{equation*}
\log _{10} \mathrm{WCE}(n, \mathbb{M}^2_{\mathrm{e}}) \leq C+
\log_{10}(\mathrm{e}^{-q\sqrt{2n}}) = C
-\sqrt{2}q\sqrt{n}\log_{10}(\mathrm{e}) \, .
\end{equation*}
This means that the logarithm of the worst case error  \eqref{eq:remainder of wce} must stay
below a line of slope
$-\sqrt{2}\frac{s}{\sqrt{\pi}}\log_{10}(\mathrm{e})$ when plotted
against $\sqrt{n}$. 

An additional difficulty arises in numerical experiments, because the
term \linebreak $\sum_{k=2n}^\infty \mathrm{e}^{-q\sqrt{k}}h_{k}(x)h_{k}(y)$ in
\eqref{eq:remainder of wce} cannot be evaluated exactly. Nonetheless,
the series converges so fast  that a simple truncation yields
sufficient  numerical accuracy. Indeed, we evaluate a truncation
of \eqref{eq:remainder of wce} for several small values of $n$ in
Figure \ref{fig:mehler 2}. Our plots show a slightly steeper slope
than $-\sqrt{2}\frac{s}{\sqrt{\pi}}\log_{10}(\mathrm{e})$ when plotted 
against $\sqrt{n}$. We must point out though that the calculations are less stable than our previous experiments in Section \ref{sec:num for Me2}, where no truncation was needed. 

\begin{figure}
\subfigure[$s_1=1$, slope $-0.50$]{
\includegraphics[width=.45\textwidth]{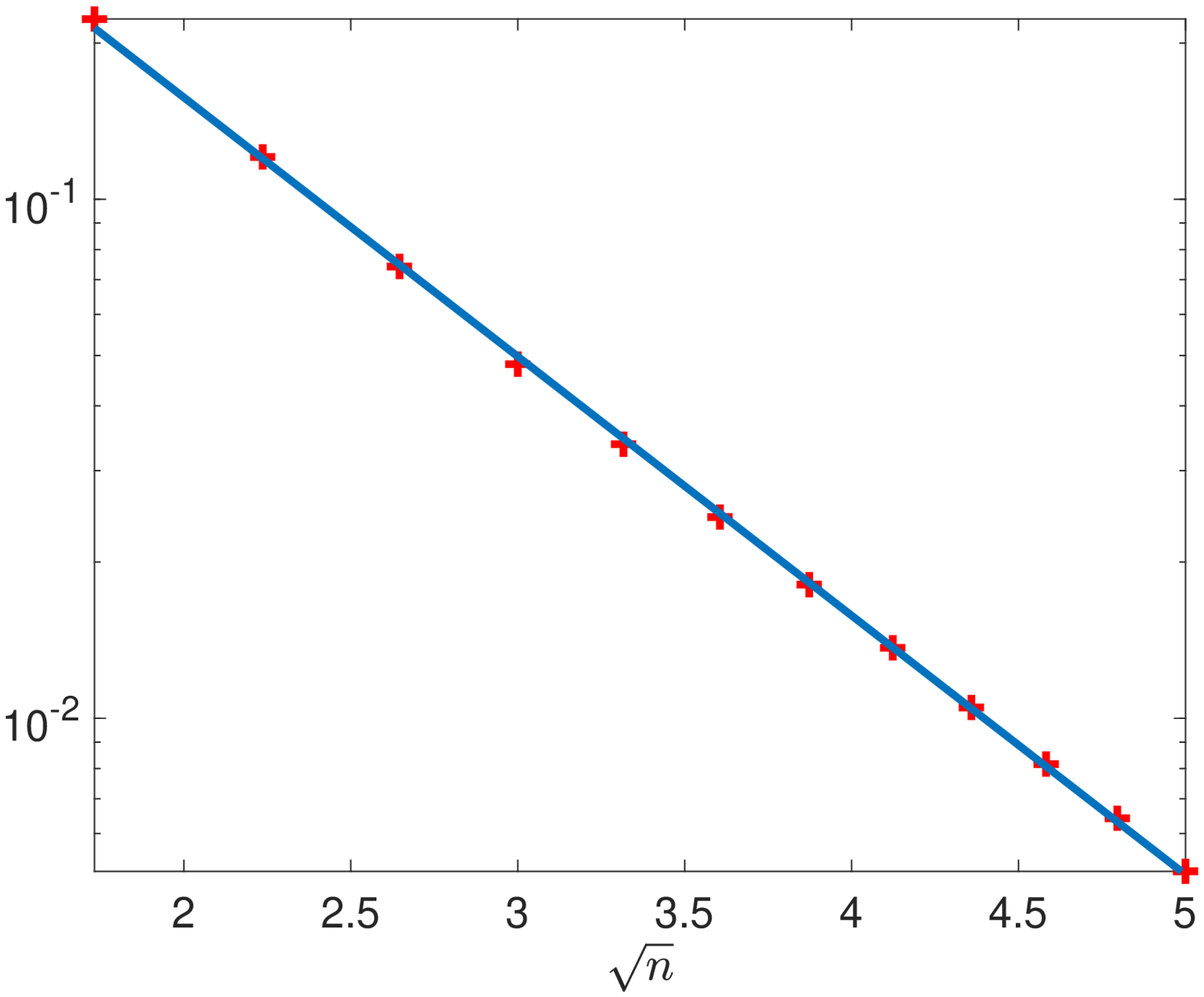}}
\subfigure[$s_2=\frac{1}{2}$, slope $-0.26$]{
\includegraphics[width=.45\textwidth]{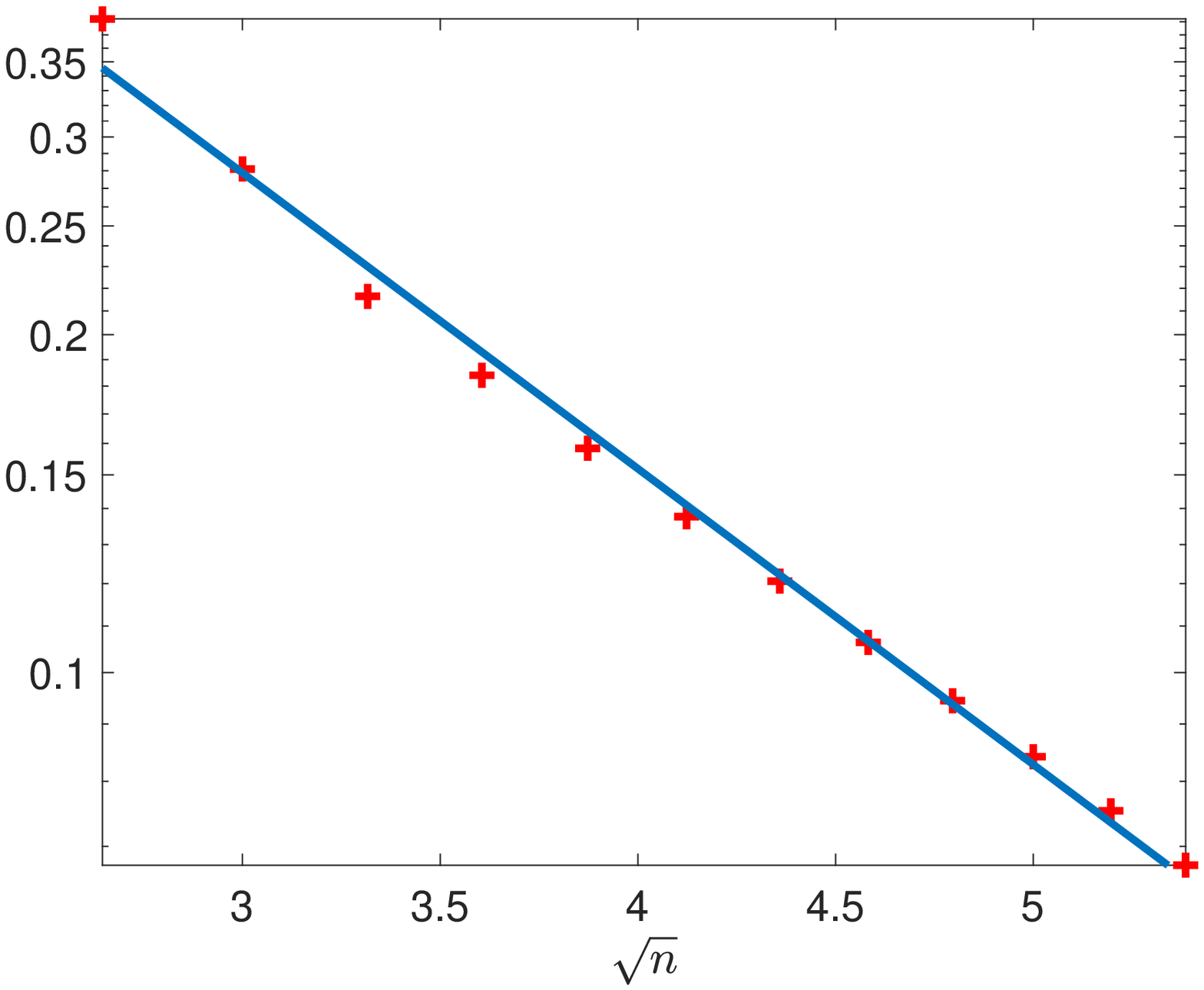}}
\caption{Logarithmic plot of the Gaussian quadrature error in $\mathbb{M}^s_{\mathrm{e}}$ with $s_1=1$ and $s_2=\frac{1}{2}$ against $\sqrt{n}$. 
The blue line is a linear least squares fit of the actual evaluation of the truncation of the worst case quadrature error \eqref{eq:remainder of wce} at $\sqrt{n}$ marked in red. 
The slope of the least squares fit is slightly steeper than $-\sqrt{2}\frac{s_1}{\sqrt{\pi}}\log_{10}(\mathrm{e})\approx -0.35$ and $-\sqrt{2}\frac{s_2}{\sqrt{\pi}}\log_{10}(\mathrm{e})\approx -0.17$, respectively.
}\label{fig:mehler 2}
\end{figure}

\section{Interlude: abstract quadrature in general Hilbert spaces}\label{sec:GHS}
In this section, we recast the quadrature problem by means of abstract
Hilbert space concepts. This reformulation will help us  to generalize
Theorem \ref{thm:alles} beyond Gaussian quadrature nodes and beyond
even parameters $\alpha$ in the Freud weight $W(x) = e^{-\pi |x|^\alpha }$. 

The general set-up is as follows: 
We assume that $\{u_k: k\in \N \}$ is an orthonormal basis for a separable Hilbert space $\mathcal{H}$ with inner product $\langle\cdot,\cdot\rangle$. For $f\in \mathcal{H}$, we define 
\begin{equation*}
\hat{f}_k:=\langle f,u_k\rangle,\qquad k\in\N.
\end{equation*}
To any sequence $\lambda=(\lambda_k)_{k\in\mathbb{N}}$ such that
$\lambda_k>0$ and $\lim \lambda _k = \infty$, we introduce the Hilbert space  
\begin{align}
\mathcal{H}_\lambda&:=\{f\in H : \sum_{k\in\N}\lambda_k |\hat{f}_k|^2  <\infty\}\label{eq:Hs2}
\end{align}
with the inner product
\begin{equation}\label{eq:inners2}
\langle f,g\rangle_{\mathcal{H}_\lambda}:=\sum_{k\in\N} \hat{f}_k \, 
\overline{\hat{g}_k} \, \lambda_k ,\qquad f,g\in \mathcal{H}_\lambda.
\end{equation}
In analogy to polynomials of degree $n$, we define the finite-dimensional subspaces 
\begin{equation*}
\Pi_n:=\spann\{u_k : 1 \leq k\leq n\} \, .
\end{equation*}
Our main assumption --- and this assumption is non-trivial --- is the
existence of (a sequence of) 
positive semi-definite sesquilinear forms $\langle \cdot ,\cdot \rangle _n$  on
$\cH _\lambda $ that are continuous (with respect to $\|\cdot\|$) 
and satisfy the following property: we  assume that there is $a_n>0$ such that 
\begin{equation}\label{eq:MZ00}
a_n \|f\|^2  \leq \|f\|^2_{n},\quad f\in\Pi_n,
\end{equation} 
where  $\|f\|_n = \langle f,f \rangle ^{1/2}_n$.  
Consequently $\langle\cdot,\cdot\rangle_n$ is an inner product on
$\Pi_n$ and we refer to it as a semi-inner product on $\cH _\lambda $. 
We may think of $\langle \cdot ,\cdot \rangle _n$ as the
discretization of some inner product via a quadrature rule of the form 
 $\langle f, g\rangle _n = \sum _{j=1}^n f(x_j)
\overline{g(x_j)} w_j$. 
The assumption~\eqref{eq:MZ00} implies that the operator $S_n: \Pi _n
\to \Pi _n$ defined by 
\begin{equation*}
S_nf =  \sum_{k=0}^n \langle f,u_k\rangle_n u_k
\end{equation*}
is invertible on $\Pi _n$. In the language of  frame theory $S_n$ is the
frame  operator associated to the basis $\{u_1,\ldots,u_n\}$ of
$\Pi _n$.

The  goal is now to  approximate the linear functional 
\begin{equation*}
\langle f,u_0\rangle
\end{equation*}
for given $f= \sum _{k=1}^\infty \hat{f}_k u_k \in \cH _\lambda $ 
by an expression containing the semi-inner product $\langle \cdot,\cdot\rangle_{n}$.  
The underlying idea is that $\langle \cdot ,\cdot \rangle _n$ is
simpler than the original inner product on $H$.

We start with some simple consequences of these definitions that are guided by analogous observations in the context of frame theory~\cite{Christensen:2003aa}, where the semi-inner products are supposed to be inner products.

\begin{lemma}\label{lemma:orth projector}
Let  $Q_n :\mathcal{H}_\lambda \rightarrow \Pi_n$ be  defined by 
\begin{equation}\label{eq:Qn}
Q_nf =
\sum_{k=0}^n \langle f,u_k\rangle_n \, S_n^{-1}u_k \, .
\end{equation}
(i) Then $Q_n$ is the orthogonal projection from $\cH _\lambda $  onto
$\Pi_n$ with respect to the semi-definite inner product
$\langle\cdot,\cdot\rangle_n$ and also satisfies  $Q_nf =\sum_{k=0}^n
\langle f,S_n^{-1}u_k\rangle_n u_k$.

(ii) $Q_nf$ is the unique solution to the optimization problem 
\begin{equation}\label{eq:argmin}
\arg\min_{p\in \Pi_n} \|f-p\|_n,\qquad \text{ for } f\in \mathcal{H}_\lambda.
\end{equation}

(iii) If $f\in \cH _\lambda $ and $g \in \Pi _n$, then
\begin{equation}
  \label{eq:cn1}
\langle Q_n f,g\rangle = \langle  f,S_n^{-1}g\rangle_n \, .  
\end{equation}
(Note the use of different inner products on both sides!)
 \end{lemma}
\begin{proof}
(i)  Define the extension of $S_n$ to $\cH _\Lambda $ by  $\tilde{S}_n:\mathcal{H}_\lambda \rightarrow \Pi_n$ by
\begin{equation*}
\tilde{S}_n f = \sum_{k=0}^n \langle f,u_k\rangle_n u_k.
\end{equation*}
Then  $Q_n =S_n^{-1}\tilde{S}_n$ is defined on all of $\cH _\Lambda
$. Its restriction to $\Pi _n$ is $Q_n|_{\Pi _n} =
S_n^{-1}\tilde{S}_n|_{\Pi _n} = S_n \inv S_n = \mathrm{Id}_{\Pi
  _n}$. Thus $Q_n$ is the  identity on $\Pi_n$, and  consequently, 
$Q_n^2 = Q_n$ and  $Q_n$ is a projection. The kernel 
of $Q_n$ equals the orthogonal complement of $\Pi_n$ with respect to
$\langle\cdot,\cdot\rangle_n$, so $Q_n$ is the orthogonal projection
with respect to $\langle \cdot ,\cdot \rangle _n$. 

Next  we observe that
$\tilde{S}_n$ and $S_n$ are self-adjoint with respect to
$\langle\cdot,\cdot\rangle_n$. 

The alternative representation of $Q_nf$ follows from $Q_nf = 0$ for
$f\in \Pi _n^\perp$ and $Q_nf =  S_n S_n \inv f = \sum _{k=1}^n
\langle S_n\inv f, u_k \rangle_n u_k$ and the self-adjointness of
$S_n\inv $ (of course, this is just the usual formalism of frame
theory~\cite{Christensen:2003aa}).

(ii) follows because  $Q_n$ is the orthogonal projection from $\cH
_\Lambda $ onto $\Pi _n$ with respect to $\langle \cdot , \cdot
\rangle _n$.

(iii) By definition of $Q_n$ and (i)  we obtain
$$
\langle Q_n f , g \rangle = \sum _{k=1}^n \langle f, S_n\inv u_n \rangle
_n \langle u_k , g\rangle = \langle f, S_n\inv g \rangle _n \, , 
$$
since $g =\sum _{k=1}^n\langle g,u_k\rangle u_k \in \Pi _n$. 
\end{proof}
Coming back to abstract quadrature, we may take $\langle Q_n f, u_0\rangle$ as an
approximation for the functional $\langle f,u_0\rangle$. In view of the identity 
\begin{equation*}
\langle Q_n f,u_0\rangle = \langle  f,S_n^{-1}u_0\rangle_n,
\end{equation*}
from \eqref{eq:cn1},   this means 
 that $\langle  f,S_n^{-1}u_0\rangle_n$ is an approximation of
 $\langle f, u_0\rangle$ that requires only  knowledge of
 $u_0,\ldots,u_n$ and the ``simpler'' inner product
 $\langle\cdot,\cdot\rangle_n$, but we do not need to evaluate
 $\langle\cdot,\cdot\rangle$ directly. In this sense $\langle
 f,S_n^{-1}u_0\rangle_n$ is an abstract quadrature rule. 
 
To quantify the approximation error, we introduce the error function 
\begin{equation} \label{errorf}
\phi_\lambda(n):=\sum_{k=n+1}^\infty  \lambda_k^{-1}\| u_k\|^2_n  \in[0,\infty].
\end{equation}
The  worst case error for the abstract quadrature rule is then given
by the following estimate.

\begin{tm}\label{thm:quadrature}
We have
\begin{equation*}
\sup_{\substack{f\in \mathcal{H}_\lambda\\ \|f\|_{\mathcal{H}_\lambda}\leq 1}}|\langle f,u_0\rangle-\langle f,S_n^{-1}u_0\rangle_n|^2  
\leq  \frac{\phi_\lambda(n)}{ a_n}.
\end{equation*} 
\end{tm}
\begin{proof}
We observe that  $Tf =  \langle f,u_0\rangle - \langle f,S_n^{-1}u_0\rangle_n$ is
a continuous linear functional  on $\mathcal{H}_\lambda$. As in the
proof of Theorem~\ref{thm:alles},   $T$ can be
identified with a vector in $\cH _\lambda $ by the Riesz representation theorem. Since
$\{\lambda_k^{-\frac{1}{2}}u_k\}_{k\in\mathbb{N}}$ is an orthonormal
basis for $\mathcal{H}_\lambda$, the  norm of the functional $T$ is
given by
$\|T\|_{\cH _\Lambda }^2 = \sum _{k\in \bN } \lambda _k\inv
|T(u_k)|^2$. This leads to 
\begin{align*}
 \|T\|_{\cH _\lambda }^2 &= \sup_{\substack{f\in \mathcal{H}_\lambda\\
  \|f\|_{\mathcal{H}_\lambda}\leq 1}} \left|\langle f,u_0\rangle  - \langle f,S_n^{-1}u_0\rangle_n \right|^2 
    \\
  &= \sum_{k=0}^\infty\lambda_k^{-1} \left|\langle u_k,u_0\rangle - \langle u_k,S_n^{-1}u_0\rangle_n  \right|^2\\
& = \sum_{k=n+1}^\infty \lambda_k^{-1} |\langle u_k,S_n^{-1}u_0\rangle_n|^2\\
& \leq \sum_{k=n+1}^\infty \lambda_k^{-1} \| u_k\|_n^2 \|S_n^{-1}u_0\|_n^2\\
&= \|S_n^{-1} u_0\|_n^2 \, \phi _\lambda (n) \, .
\end{align*}
In this chain of identities we have used \eqref{eq:cn1} in the form
$\langle u_k, S_n\inv u_0 \rangle _n = \langle Q_n u_k, u_0 \rangle =
\langle u_k, u_0\rangle = \delta _{k,0}$ for $k \leq n$.

To derive the bound $\|S_n^{-1}u_0\|_n^2\leq a_n^{-1}$, we recall
that $S_n$ is self-adjoint with respect to~$\langle
\cdot,\cdot\rangle_n$ and positive semi-definite. Hence, the same
holds for $S_n^{-1}$ and its square root.  Using \eqref{eq:cn1} again for $g\in \Pi_n$, we obtain 
\begin{align*}
\|S_n^{-1/2} g\|_n^2 = \langle S_n^{-1/2} g,S_n^{-1/2}g\rangle_n =
  \langle g,S_n^{-1}g\rangle_n = \langle g,g\rangle=\|g\|^2\leq
  a_n^{-1}\|g\|_n^2 \, .
\end{align*}
This implies  
\begin{equation*}
\|S_n^{-1}u_0\|_n^2 \leq a_n^{-1} \|S_n^{-1/2}u_0\|_n^2 =  a_n^{-1} \|u_0\|^2=a_n^{-1},
\end{equation*}
which completes the proof.
\end{proof}
\begin{rem}
We may replace $u_0$ with any element $g\in \Pi_n$ and obtain
\begin{equation*}
\sup_{\substack{f\in \mathcal{H}_\lambda\\ \|f\|_{\mathcal{H}_\lambda}\leq 1}}|\langle f,g\rangle-\langle f,S_n^{-1}g\rangle_n|^2 \leq  \frac{\phi_\lambda(n)}{a_n}\|g\|^2.
\end{equation*} 
If $a_n=1$ and \eqref{eq:MZ00} is replaced by the identity $\|f\|^2 =
\|f\|_n^2$, then $S_n$ is the identity operator, and  
\begin{equation*}
\sup_{\substack{f\in \mathcal{H}_\lambda\\ \|f\|_{\mathcal{H}_\lambda}\leq 1}}|\langle f,g\rangle-\langle f,g\rangle_n|^2 \leq  \phi_\lambda(n)\|g\|^2.
\end{equation*} 
\end{rem}
\section{The use of MZ-inequalities in quadrature}\label{sec:MZ}
In this section we extend the error estimates of
Theorem~\ref{thm:alles} considerably so that  they hold (i) for more
general point sets  beyond Gaussian quadrature nodes and (ii)  for the
full range of Freud weights $W(x) = e^{-\pi |x|^\alpha }$ for all
$\alpha> 1$. For this we will apply the results  about abstract
quadrature rules from Section~\ref{sec:GHS}.

Given the weight $W(x) = e^{-\pi |x|^\alpha }$ and the associated
orthogonal polynomials $H_k, k=0,\dots $, we choose the orthonormal
basis 
\begin{equation*}
u_k=h_{k}= H_k W, \quad k\in\mathbb{N}
\end{equation*} 
as the orthonormal basis for $L^2(\bR )$. Let us emphasize again  that
$h_k$ depends on $\alpha$.
The sequence of weights for the definition of the associated Sobolev
spaces $\cH _\lambda $ are either of  polynomial type $\lambda _k = (1+k)^s$ or  of
exponential type $\lambda _k = e^{qk^p}$. At least for $\alpha = 2$, these
abstractly defined spaces coincide with a class of well-known
function spaces in analysis, see Section~\ref{sec:modulation}. 

The associated
sequence of finite-dimensional subspaces are the spaces  $\Pi _n =
\mathrm{span} \{1, x, \dots , x^n\} W$ consisting of weighted
polynomials of degree $\leq n$. 

Next we consider a sequence of nodes $X_n\subseteq \bR $ and of
nonnegative weights $\{\tau(x)\}_{x\in X_n}$ and for every $n$ define
a semi-inner product by 
\begin{equation}\label{eq:alpha inner}
\langle f,g\rangle_{n}=\sum_{x\in X_n} \tau(x) f(x)g(x),\qquad \|f\|_{n}^2 = \langle f,f\rangle_{n}.
\end{equation}
Here is the decisive condition: $X_n$ must be a \mz\ family
and  every $X_n$ must satisfy a \mz\
inequality (in different terminology: a sampling inequality), i.e.,
there exist  $0<a_{n}\leq b_{n}<\infty$ such that  
\begin{equation} \label{eq:MZ new}
a_{n} \|f\|^2_{L^2(\mathbb{R})}\leq \sum _{x\in X_n} |f(x)|^2 \tau (x)
= \|f\|_{n}^2\leq b_{n}\|f\|^2_{L^2(\mathbb{R})},\qquad \text{ for all
} f\in \Pi_{n} \, .
\end{equation}
Assumption~\eqref{eq:MZ new} is non-trivial. In view of the vast
literature on \mz\ inequalities we take the existence of \mz\ families
for granted. For general constructions, we refer to~\cite{Filbir:2011fk,Lubinsky:1996aa,Mhaskar:2002ys}, for a
simple direct derivation of quadrature rules from \mz\ inequalities see~\cite{Grochenig:2019mq}.

To apply Theorem~\ref{thm:quadrature} we need one more property of the
orthogonal polynomials associated to the Freud weights. Let $m_n$ be 
the Mhaskar-Rahmanov-Saff numbers defined by 
\begin{equation*}
m_{n}=m_{n,\alpha }= \frac{2}{\sqrt{\pi}}\left(\frac{\Gamma(\frac{\alpha}{2})^2}{4\Gamma(\alpha)}\right)^{\frac{1}{\alpha}}
n^{\frac{1}{\alpha}}, 
\end{equation*}
cf.~\cite{Levin:1998wj}. Then, according to \cite{Levin:1992mb}, there
is $L>0$ such that the zeros of $h_n$, which are the  nodes for Gauss
quadrature chosen in  Section \ref{sec:GQ},  are contained in an
interval bounded by $m_{n}(1+Ln^{-\frac{2}{3}})$. It is therefore
natural to assume that a \mz\ family used for a quadrature rule should
satisfy a similar restriction. We may, for instance, think of $X_n$ to
be distorted Gaussian nodes, so that the following result can  be read
 as a stability result for Gaussian quadrature. 

We now state the generalization of  Theorem~\ref{thm:alles} beyond Gaussian nodes and beyond even $\alpha$.
\begin{tm}\label{tm:H with MZ}
Let $W(x) = e^{-\pi |x|^\alpha }$ for $\alpha >1$ arbitrary, and let
$\mathbb{H}^s$ and $\mathbb{E}^p_q$ be the associated Sobolev 
spaces. Suppose that the \mz\ inequalities \eqref{eq:MZ new} hold for
$X_n\subseteq \bR $  and assume that, for some $L>0$,  
\begin{equation}\label{eq:nodes hermite still bounded}
\max _{x\in X_n} |x|\leq m_{n,\alpha}(1+Ln^{-\frac{2}{3}}) \, .
\end{equation}
Define the quadrature weights 
\begin{equation}\label{eq:omega from tau}
\omega(x)=\tau(x)(S_{n}^{-1}W)(x),\qquad x\in X_n \, .
\end{equation}
Then we have, for $s>1-\frac{1}{\alpha}$ and $p,q>0$,
\begin{align*}
\sup_{\substack{f\in \mathbb{H}^s\\ \|f\|_{\mathbb{H}^s}\leq 1}}\left|\int_{-\infty}^\infty f(x)W(x)\mathrm{d}x-\sum_{x\in X_n}\omega(x) f(x) \right|^2  & \lesssim   \frac{b_{n}}{a_{n}} n^{-s+\frac{4}{3}},\\
\sup_{\substack{f\in \mathbb{E}^{p}_q\\ \|f\|_{\mathbb{E}^{p}_q}\leq 1}}\left|\int_{-\infty}^\infty f(x)W(x)\mathrm{d}x-\sum_{x\in X_n} \omega(x) f(x) \right|^2 &\lesssim   \frac{b_{n}}{a_{n}} \mathrm{e}^{-q n^p} \cdot 
n^{\frac{2}{3}+\max(1-p,0)}
\end{align*}
with a constant independent of $n$. 
\end{tm}
\begin{proof}
According to Theorem \ref{thm:quadrature}, we need  to derive suitable
bounds for the error function $\phi$ defined in~\eqref{errorf} for the
weight sequences $\lambda _k = (1+k)^s$ and $\lambda _k=
e^{qk^p}$.  For $\mathbb{H}^s$ and $\mathbb{E}^p_q$ these are 
\begin{align*}
\phi^s(n) & =\sum_{k=n+1}^\infty (1+k)^{-s} \sum_{x\in X_n}\tau(x) |h_{k}(x)|^2 \\
\phi^{p}_q(n)& = \sum_{k=n+1}^\infty \mathrm{e}^{-q k^p} \sum_{x\in X_n}\tau(x) |h_k(x)|^2.
\end{align*}
By interchanging the summation and applying Proposition \ref{prop:new one again}, we obtain
\begin{align}
\phi^s(n)  & =\sum_{x\in X_n}\tau(x) \sum_{k=n+1}^\infty (1+k)^{-s} 
             |h_{k}(x)|^2 \lesssim  \Big( \sum_{x\in X_n}\tau(x) \Big) \,  n^{-s+1-\frac{1}{\alpha}},\label{eq:phi 01}\\
\phi^{p}_q(n)& =\sum_{x\in X_n}\tau(x) \sum_{k=n+1}^\infty
               \mathrm{e}^{-q k^p}  |h_k(x)|^2\lesssim  \Big(\sum_{x\in
               X_n}\tau(x)\Big) \,   \mathrm{e}^{-q n^p} \cdot  n^{\frac{1}{3}-\frac{1}{\alpha}+\max(1-p,0)}.\label{eq:phi 02}
\end{align}
 We still  need a  bound for $\sum_{x\in X_n}\tau(x)$. 
We apply the \mz\ inequality \eqref{eq:MZ new} to each $h_k, k\leq n,$ and
recall  that $\sum _{k=0}^n h_k^2 = \Lambda _n\inv$ is the reciprocal of
the Christoffel function. We obtain 
$$
\sum_{x\in X_n}\tau(x)h_k(x)^2 \leq b_{n}
\int_{-\infty}^\infty h_k(x)^2\mathrm{d}x  =b_{n},
$$
so that 
\begin{equation}\label{eq:upper MZ used here}
\sum_{x\in X_n}\tau(x)\Lambda^{-1}_{n}(x) \leq b_{n} \int_{-\infty}^\infty \Lambda^{-1}_{n}(x)\mathrm{d}x  =b_{n}(n+1).
\end{equation}
Consequently
\begin{align} \label{eq:sum of tau weights} 
\sum_{x\in X_n}\tau(x)  =  \sum_{x\in X_n}\tau(x) \Lambda^{-1}_n(x)\Lambda_n(x)
 \leq  b_{n}(n+1) \sup_{x\in X_n}\Lambda_n(x).
\end{align}
With the hypothesis $|x|\leq m_{n} (1+Ln^{-\frac{2}{3}})$ in~\eqref{eq:nodes hermite still bounded}, the estimates
for $\Lambda _n$   
from \cite{Levin:1992mb} imply that 
\begin{equation}\label{eq:christ for hermite}
\Lambda_{n}(x) \lesssim  n^{\frac{1}{\alpha}-\frac{2}{3}},
\end{equation}
where the constant may depend on $\alpha$ and $L$. Hence, we obtain
\begin{equation}\label{eq:cn bound}
\sum_{x\in X_n}\tau(x) \lesssim  b_{n} n^{\frac{1}{\alpha}+\frac{1}{3}},
\end{equation}
which leads to 
\begin{align*}
\phi^s(n) & \lesssim b_{n} n^{-s+\frac{4}{3}},\\
\phi^{p}_q(n)& \lesssim b_{n} \mathrm{e}^{-q n^p} \cdot  n^{\frac{2}{3}+\max(1-p,0)}.
\end{align*}
This concludes the proof since all assumptions of Theorem
\ref{thm:quadrature} are satisfied. 
\end{proof}
A few remarks are in order. 
\begin{rem}
(i) In the literature it is usually assumed that the \mz\ inequalities
 hold with uniform constants. In this case  there are $a$ and $b$ such that 
\begin{equation*}
0<a\leq a_{n}\leq b_{n}\leq b<\infty,
\end{equation*}
and the fraction $\frac{b_{n}}{a_{n}}\leq \frac{b}{a}$ dissolves into
the hidden  constant of $\lesssim$.

(ii) The weights in the \mz\ inequalities  $\tau(x)$ for $x\in X_n$
are nonnegative by assumption. In general, however, the quadrature
weights  $\omega(x)$ defined by \eqref{eq:omega from tau} could possibly
be negative. 
However, due to continuity arguments, if $X_n$ are distorted Gaussian nodes, then $\{\omega(x)\}_{x\in X_n}$ are still positive for sufficiently small distortions. 

(iii) The case of Gauss quadrature can be viewed as a special case of
Theorem~\ref{tm:H with MZ}. In this case the nodes for $X_{n+1}$ are  the
zeros of $h_{n+1}$. 
Endowed with the weights
\begin{equation}\label{eq:tau general via Gauss}
\tau(x)=\Lambda_n(x),\quad x\in X_{n+1},
\end{equation}
they satisfy the \mz\ inequalities with equality $a_n=b_n=1$.
Indeed, if $f=pW \in \Pi _n$ with a polynomial  of degree $n$, then
$|p|^2W \in \Pi _{2n}$ and the exactness of $X_{n+1} $ on $\Pi _{2n}$
with $\omega (x) = \Lambda _n(x) W(x)$ means that
$$
\sum _{x\in X_{n+1}} |p(x)|^2 W(x) \omega (x) = \int _{\bR } |p(x)|^2
  W(x)^2 \, \mathrm{d}x \, .
  $$
 Written in terms of $f\in \Pi _n$, this is the \mz\ inequality
  $$
\sum _{x\in X_{n+1}} |f(x)|^2 \Lambda _n(x)W(x)^{-2} = \int _{\bR } |f(x)|^2
  \, \mathrm{d}x \, ,
  $$
  with constants $a_n=b_n=1$. 
Therefore,
$S_{n}$ is the identity operator, and the quadrature weights
$\{\omega(x)\}_{x\in X_{n+1}}$ defined by  \eqref{eq:omega from tau}
coincide with the Christoffel weights in \eqref{eq:omega through
  Christoffel}. Thus, Theorem \ref{tm:H with MZ} contains Gauss
quadrature as a special case, but 
it yields slightly weaker  bounds than Theorem
\ref{thm:alles}. 

(iv) The proof of Theorem \ref{tm:H with MZ} reveals two key
ingredients for deriving quadrature bounds. We need
\begin{itemize}
\item[-] upper and lower \mz\ inequalities, i.e., sampling theorems for
  the finite-dimensional subspaces $\Pi _n$, and 
\item[-] upper and lower bounds on Christoffel functions.
\end{itemize}
Lower \mz\ inequalities are used in the abstract Theorem
\ref{thm:quadrature} to reduce the problem of finding suitable bounds
on $\phi^s$ and $\phi^{p}_q$. The lower Christoffel bounds are used in
the proof of Proposition \ref{prop:new one again} to
ensure~\eqref{eq:phi 01} and \eqref{eq:phi 02} hold. Then upper 
\mz\ inequalities in \eqref{eq:upper MZ used here} combined with upper
Christoffel bounds in \eqref{eq:christ for hermite} are used to bound
the sum of quadrature weights in \eqref{eq:sum of tau weights}.

(v) Optimality: How sharp are the estimates of
Theorems~\ref{thm:alles} and \ref{tm:H with MZ}?

For the weight function $W(x) = e^{-\pi
  x^2}$ ($\alpha =2$) and the Sobolov spaces $\mathbb{H}^s$ it was
shown in \cite{Dick:2018aa} that  
no choice of points and weights can do better than $\cO (n^{-2s})$. This bound is achieved by points derived from digital nets up to some logarithmic factor (at least when $n$ is a power of a prime). 

For quadrature in $\mathbb{E}^{p}_q$ with $\alpha=2$ and $p\geq 1$,
according to \cite{Irrgeher:2015ab}, no choice of points and weights
can do better than $\mathrm{e}^{-q(2n)^p} 4^{-n}n^{-2}$.

In view of the generality of our techniques, our results come fairly
close to the  results that are optimal in special cases. 

(vi) Theorem \ref{tm:H with MZ} holds for a more general class of
Borel measures instead of just point measures. We assume that $\nu _n$
is a sequence of Borel measures that define the semi-inner products 
\begin{equation*}
\langle f,g\rangle_{n} = \int_{\mathbb{R}} f(x)
g(x)\mathrm{d}\nu_{n}(x) \, ,
\end{equation*}
such that (a) the \mz\ inequalities  \eqref{eq:MZ new} are satisfied
and (b) the support condition \begin{equation*}
\max _{x\in\supp(\nu_{n})} |x|\leq m_{n}(1+Ln^{-\frac{2}{3}})
\end{equation*}
holds for some $L>0$. Then the conclusion of Theorem \ref{tm:H with
  MZ} holds for the quadrature rule $\int_{-\infty}^\infty
f(x)(S_{n}^{-1}W)(x)\mathrm{d}\nu_{n}(x)$.  
\end{rem}

\textbf{Numerics.} 
Fix $\alpha =2$ and $W(x) = e^{-\pi x^2}$ and let $X_{n+1}$ be  the Gaussian nodes  with weights $\{\tau(x)\}_{x\in
  X_{n+1}}$ in \eqref{eq:tau general via Gauss}.  We
now choose perturbed  nodes $\tilde{X}_{n+1}$ by  
\begin{equation}\label{eq:distorted Gaussian nodes}
\tilde{x} := x+\epsilon(x),\qquad x\in X_{n+1}.
\end{equation}
If the perturbations $\epsilon(x)$ are sufficiently small, then the 
$\tilde{X}_{n+1} $ still satisfy the \mz\ inequalities \eqref{eq:MZ new}
with the same weights $\{\tau(x)\}_{x\in X_{n+1}}$, e.g., by~\cite{Christensen:2003aa}.
 According to \eqref{eq:omega from tau}, we define
\begin{equation*}
\omega(\tilde{x})=\tau(x)(S_{n}^{-1}W)(\tilde{x}),\quad x\in X_{n+1}.
\end{equation*}
The quadrature error in $\mathbb{M}^s_{\mathrm{e}}$ is determined analogously to Section \ref{sec:Me1}, so that the proof of Theorem \ref{thm:alles} leads to 
\begin{multline}\label{eq:wce Me1 distorted}
\sup_{\substack{f\in \mathbb{M}^s_{\mathrm{e}}, \\ \|f\|_{\mathbb{M}^s_{\mathrm{e}}}\leq 1}} \left| \int\limits_{\,-\infty}^\infty f(x) W(x) \mathrm{d}x- \sum_{\tilde{x}\in \tilde{X}_{n+1}}\omega(\tilde{x}) f(\tilde{x})\right|^2 \\
\asymp\sum_{\tilde{x},\tilde{y}\in \tilde{X}_{n+1}} \omega(\tilde{x})\omega(\tilde{y}) \sum_{k=n+1}^\infty \mathrm{e}^{-\frac{s}{\sqrt{\pi}}\sqrt{k}} h_{k}(\tilde{x})h_{k}(\tilde{y}).
\end{multline}
For the quadrature error in $\mathbb{M}^s$, we obtain
\begin{multline}\label{eq:wce M distorted}
\sup_{\substack{f\in \mathbb{M}^s, \\ \|f\|_{\mathbb{M}^s}\leq 1}} \left| \int\limits_{\,-\infty}^\infty f(x) W(x) \mathrm{d}x- \sum_{\tilde{x}\in \tilde{X}_{n+1}}\omega(\tilde{x}) f(\tilde{x})\right|^2 \\
\asymp\sum_{\tilde{x},\tilde{y}\in \tilde{X}_{n+1}}  \omega(\tilde{x})\omega(\tilde{y}) \sum_{k=n+1}^\infty (1+k)^{-s} h_{k}(\tilde{x})h_{k}(\tilde{y}).
\end{multline}
For numerical experiments  in Figure \ref{fig:mehler 3} we
have  truncated the infinite series in \eqref{eq:wce Me1 distorted}
and \eqref{eq:wce M distorted} and then plotted the worst case error
as a function of $n$ and $\sqrt{n}$ respectively. For $\mathbb{M}^{\frac{1}{2}}_{\mathrm{e}}$ in particular, we observed some stability issues. For $\mathbb{M}^s$ nonetheless, the results suggest that the actual decay rate is $n^{-s}$ and the additional $\frac{4}{3}$ in Theorem \ref{tm:H with MZ} is simply an artifact of our proof. 
\begin{figure}
\subfigure[$\mathbb{M}^{\frac{1}{2}}_{\mathrm{e}}$, slope $-0.26$ ]{\includegraphics[width=.45\textwidth]{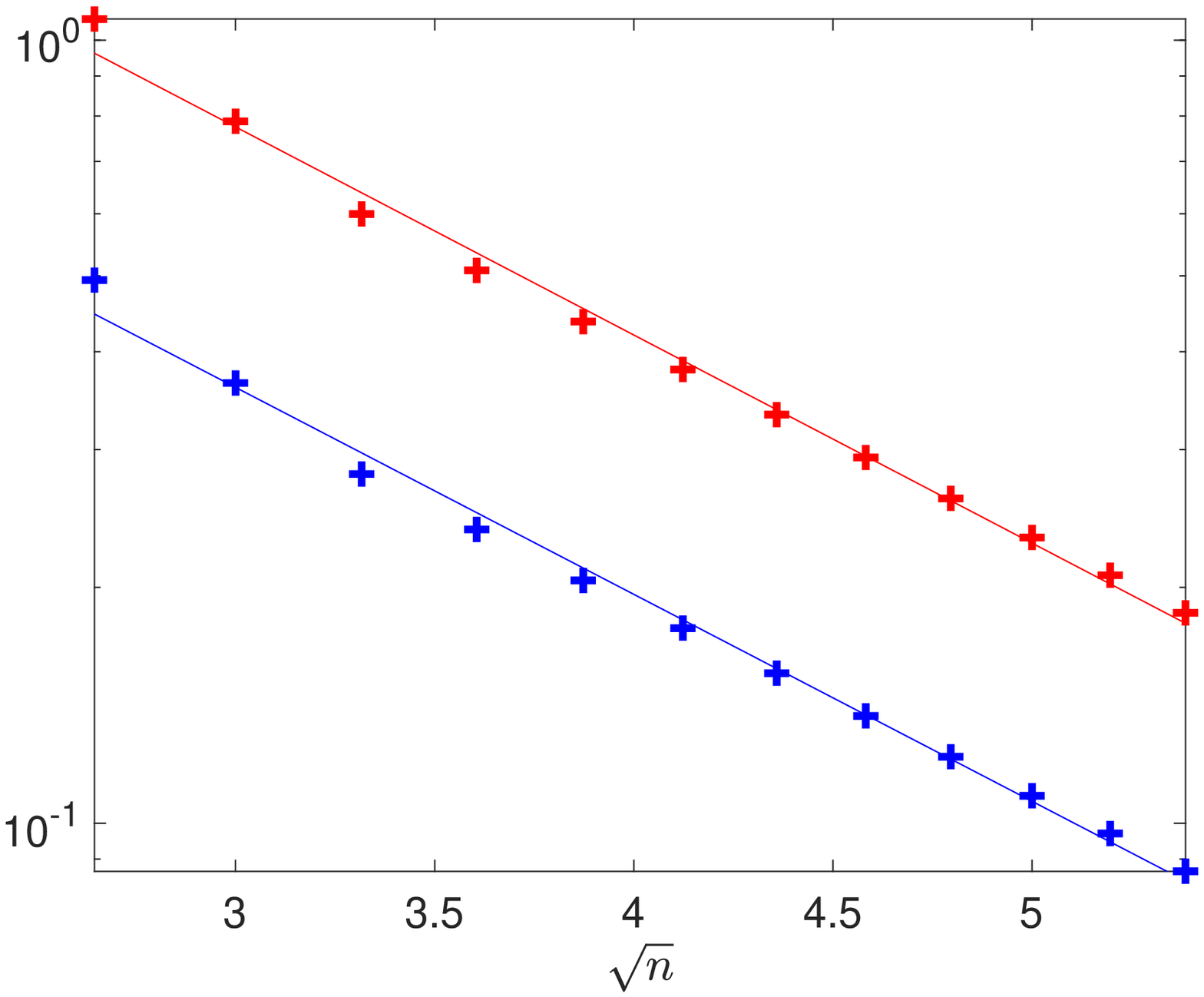}}

\subfigure[$\mathbb{M}^1$, slope $-1$ ]{\includegraphics[width=.45\textwidth]{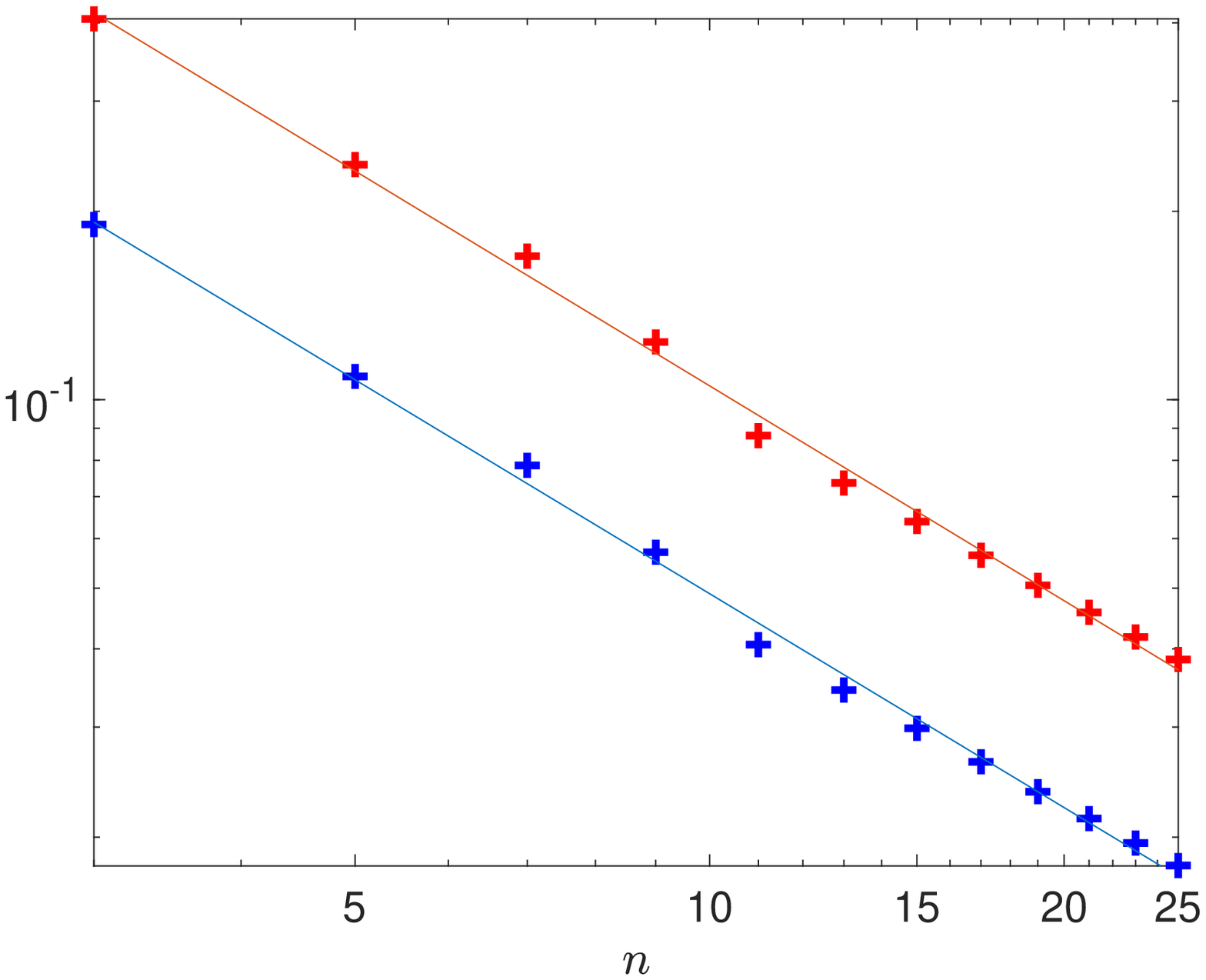}}
\subfigure[$\mathbb{M}^{\frac{2}{3}}$, slope $-.66$ ]{\includegraphics[width=.45\textwidth]{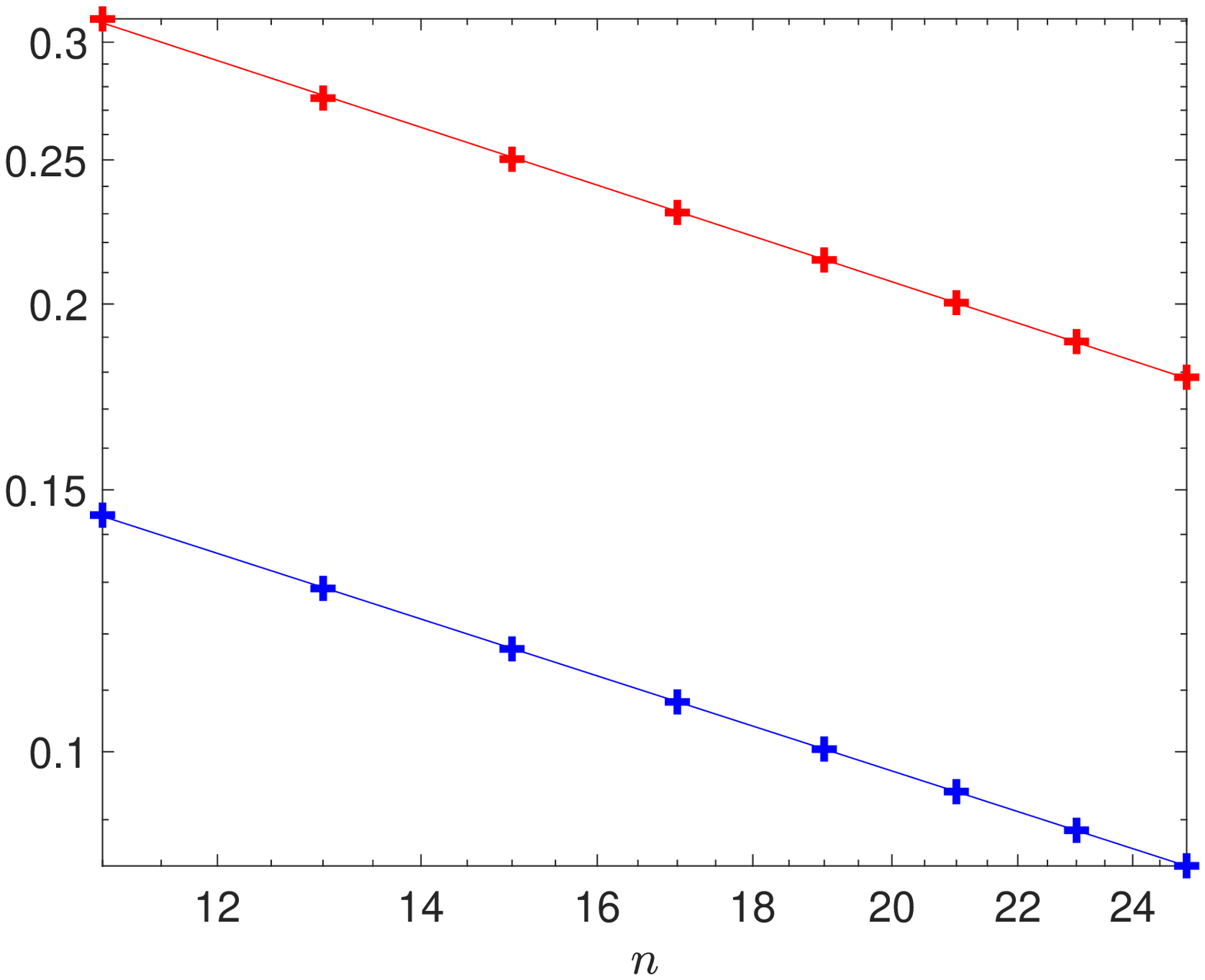}}
\caption{Logarithmic plot of the quadrature errors \eqref{eq:wce Me1 distorted} and \eqref{eq:wce M distorted} for distorted Gaussian nodes against $\sqrt{n}$ and $n$, respectively. 
 (blue) $|\epsilon(x)|=\frac{1}{10}$, (red) $|\epsilon(x)|=\frac{1}{5}$ in \eqref{eq:distorted Gaussian nodes}.}\label{fig:mehler 3}
\end{figure}

\appendix

\section{Multivariate setting}
Using tensor products the results of the previous  sections can be
extended to several variables. 
Fix  $\alpha>1$, write $x=(x_1, \dots , x_d) \in \rd $, and define the
tensor weight $W_d$ on $\rd $ as 
\begin{equation*}
W_d(x)=W(x_1)\cdots W(x_d),\qquad x\in\mathbb{R}^d, 
\end{equation*}
where $W(x_i)=e^{-\pi |x_i|^{\alpha}}$. 
The tensor product of the generalized Hermite functions is denoted by
\begin{equation*}
h_{k}(x)=h_{k_1}(x_1)\cdots h_{k_d}(x_d),\qquad x\in\mathbb{R}^d,\quad k\in\mathbb{N}^d,
\end{equation*}
and we still keep the notation $\hat{f}_{k}=\langle
f,h_{k}\rangle_{L^2(\mathbb{R}^d)}$. Then $\{ h_k : k\in \bN ^d\} $ is
an orthonormal basis for $\lrd $.

 For $s\geq 0$ and $p,q> 0$, the  multivariate modulation spaces
 $\mathbb{H}^s(\mathbb{R}^d)$ and $\mathbb{E}^{p}_q(\mathbb{R}^d)$  are defined
 by the norms 
\begin{align*}
 \|f\|^2_{\mathbb{H}^s(\mathbb{R}^d)}&=\sum_{k\in\mathbb{N}^d} (1+k_1)^s\cdots(1+k_d)^s |\hat{f}_{k}|^2,\\
 \|f\|^2_{\mathbb{E}^{p}_q(\mathbb{R}^d)}& =  \sum_{k\in\mathbb{N}^d}^\infty
  \mathrm{e}^{-q(k_1^p+\ldots+k_d^p)}  |\hat{f}_{k}|^2 \, .
\end{align*}

We choose the Cartesian product of the Gaussian quadrature nodes
$X_n$, i.e., $X_{n,d} = X_n \times \dots \times  X_n = \{x \in \rd:
x_j\in X_n\}$ and the quadrature weights $\omega _d(x) = \prod
_{j=1}^d \omega (x_j), x_j\in X_n$, where $\omega $ is taken from
\eqref{eq:omega through Christoffel}. Then   
for $s>1-\frac{1}{\alpha}$ and $p,q>0$, one derives
\begin{align}
\sup_{\substack{f\in \mathbb{H}^s(\mathbb{R}^d)\\ \|f\|_{\mathbb{H}^s}\leq 1}}\left|\int_{\mathbb{R}^d} f(x)W(x)\mathrm{d}x-\sum_{x\in X_{n}} \omega(x) f(x)\right|^2   &\lesssim   n^{-s+1-\frac{1}{\alpha}},\label{eq:11011}\\
\sup_{\substack{f\in \mathbb{E}^{p}_q(\mathbb{R}^d)\\ \|f\|_{\mathbb{E}^{p}_q}\leq 1}}\left|\int_{\mathbb{R}^d} f(x) W_\alpha(x)\mathrm{d}x-\sum_{x\in X_n}\omega(x) f(x) \right|^2 &\lesssim \mathrm{e}^{-q(2n)^p} \cdot n^{\frac{1}{3}-\frac{1}{\alpha}}.\label{eq:11012}
\end{align}
Note that $X_{n,d}$ consists of $N=(n+1)^d$ points, so that the
asymptotic error is actually $\cO (N^{-(\frac{s}{d}+\frac{1}{d}-\frac{1}{\alpha d}})$ in dimension
$d$ for $\mathbb{H}^s(\mathbb{R}^d)$.  We may also consider different $\alpha_1,\ldots,\alpha_d$ for each dimension in the tensor product. Then \eqref{eq:11011} and \eqref{eq:11012} hold for $\alpha$ replaced by $\alpha_{\max}=\max(\alpha_1,\ldots,\alpha_d)$ for all $s>1-\frac{1}{\alpha_{\max}}$.

For the normalized multivariate Gaussian window function 
\begin{equation*}
\varphi_d(x)=2^{\frac{d}{4}}e^{-\pi \|x\|^2},\qquad x\in\mathbb{R}^d,
\end{equation*}
the short-time Fourier transform of $f\in L^2(\mathbb{R}^d)$ 
is 
\begin{equation*}
V_{\varphi_d} f(x,\xi) = \int_{\mathbb{R}^d} f(t) \varphi_d(t-x) e^{- 2 \pi \mathrm{i} \langle \xi, t\rangle}
dt,\qquad x,\xi\in\mathbb{R}^d.
\end{equation*}
For $f\in L^2(\mathbb{R}^d)$, the modulation norms
\begin{align*}
\|f\|^2_{\mathbb{M}^s(\mathbb{R}^d)} &= \int _{\bR ^{2d}} |V_{\varphi_d} f(z)|^2  (1+|z_1|^2)^s\cdots(1+|z_d|^2)^s\, \mathrm{d}z\\
\|f\|^2_{\mathbb{M}^s_{\mathrm{e}}(\mathbb{R}^d)} &= \int _{\bR ^{2d}} |V_{\varphi_d} f(z)|^2 \, \mathrm{e}^{s\|z\|_1}\mathrm{d}z\\
\|f\|^2_{\mathbb{M}^s_{\mathrm{e}^2}(\mathbb{R}^d)} &= \int _{\bR ^{2d}} |V_{\varphi_d} f(z)|^2 \, \mathrm{e}^{s\|z\|_2}\mathrm{d}z,
\end{align*}
where $\|\cdot\|_1$ and $\|\cdot\|_2$ denote the standard $1$ and $2$-norm on $\mathbb{C}^d$, 
lead to the respective multivariate modulation spaces $\mathbb{M}^s(\mathbb{R}^d)$, $\mathbb{M}^s_{\mathrm{e}}(\mathbb{R}^d)$, and $\mathbb{M}^s_{\mathrm{e}^2}(\mathbb{R}^d)$. Since 
\begin{equation*}
V_{\varphi_d} h_k = V_{\varphi_1} h_{k_1} \otimes \ldots \otimes V_{\varphi_1} h_{k_d},\qquad k\in\mathbb{N}^d,
\end{equation*}
they carry the tensor structure, and one may deduce, for $\alpha=2$, 
\begin{align*}
\mathbb{M}^s(\mathbb{R}^d)&=\mathbb{H}^{s}(\mathbb{R}^d),\qquad s\geq 0,\\
\mathbb{M}^s_{\mathrm{e}}(\mathbb{R}^d) & =\mathbb{E}^{\frac{1}{2}}_q(\mathbb{R}^d),\qquad q=\frac{s}{\sqrt{\pi}},\quad s\geq 0,\\
\mathbb{M}^s_{\mathrm{e}^2}(\mathbb{R}^d) & =\mathbb{E}^{1}_q(\mathbb{R}^d),\qquad q=\ln(\tfrac{\pi}{\pi-s}),\quad \pi>s\geq 0.
\end{align*}
with equivalent norms.



\end{document}